\documentclass[review,hidelinks,onefignum,onetabnum]{siamart251216}

\usepackage{lipsum}
\usepackage{amsfonts}
\usepackage{graphicx}
\usepackage{epstopdf}
\usepackage{algorithmic}
\ifpdf
  \DeclareGraphicsExtensions{.eps,.pdf,.png,.jpg}
\else
  \DeclareGraphicsExtensions{.eps}
\fi

\newsiamremark{remark}{Remark}
\newsiamremark{hypothesis}{Hypothesis}
\crefname{hypothesis}{Hypothesis}{Hypotheses}
\newsiamthm{claim}{Claim}
\newsiamremark{fact}{Fact}
\crefname{fact}{Fact}{Facts}

\headers{DECOUPLING SCALES VIA LSI AND SPLITTING}{Eric T. Chung, Lijian Jiang, Mengnan Li and Yajun Wang}

\title{Decoupling scales via localized subspace iteration and temporal splitting for multiscale parabolic equations \thanks{Submitted to the editors DATE.
\funding{Eric T. Chung's work is partially supported by the Hong Kong RGC General Research Fund (Project Nos. 14304525 and 14305624). Lijian Jiang acknowledges the support of the NSFC (Project No. 12271408). Mengnan Li acknowledges the support of the NSFC (Project No. 12401568) and the Basic Research Program of Jiangsu, China (Project No. BK20230346).}}}

\author{
  Eric T. Chung\thanks{Department of Mathematics, The Chinese University of Hong Kong, Hong Kong SAR
  (\email{eric.t.chung@cuhk.edu.hk}).}
\and Lijian Jiang \thanks{School of Mathematical Sciences, Tongji University, Shanghai, P.R. China
  (\email{ljjiang@tongji.edu.cn}).}
\and Mengnan Li \thanks{College of Science, Nanjing University of Posts and Telecommunications, Nanjing, P.R. China
 (\email{mengnanli@njupt.edu.cn}).}
\and Yajun Wang \thanks{Department of Mathematics, The Chinese University of Hong Kong, Hong Kong SAR
  (\email{yajunwang@cuhk.edu.hk}).}
  }

\usepackage{amsopn}

\usepackage{amsmath}
\usepackage{amssymb}
\usepackage{microtype}
\usepackage{subcaption}
\usepackage{graphicx}
\usepackage{booktabs}
\usepackage{multirow}
\usepackage{verbatim}

\ifpdf
\hypersetup{
  pdftitle={Decoupling scales via localized subspace iteration and temporal splitting for multiscale parabolic equations},
  pdfauthor={Eric T. Chung, Lijian Jiang , Mengnan Li and Yajun Wang}
}
\fi

\begin{document}

\maketitle

\begin{abstract}
Simulating diffusion in heterogeneous media presents a significant computational challenge, as resolving microscopic physical scales traditionally demands excessively fine computational grids. To overcome this barrier, we extend the Localized Subspace Iteration (LSI) framework to multiscale parabolic equations. The proposed method constructs optimal, low-dimensional trial spaces by iteratively approximating the dominant eigenspaces of local inverse operators via Localized Standard Subspace Iteration (LSSI) or Localized Krylov Subspace Iteration (LKSI). Because these LSI basis functions are inherently tailored to capture the slow-decaying, low-frequency modes of the parabolic solution, they naturally suppress error accumulation over long-term integration. To further improve computational efficiency, we decouple the basis construction into an offline phase and implement a contrast-independent, partially explicit temporal splitting scheme for online time-stepping. By explicitly advancing the dominant macroscopic modes while implicitly treating high-frequency microscopic corrections, this scheme guarantees stability without imposing restrictive time-step constraints. We establish rigorous a priori error estimates in both the energy and $L^2$ norms. Numerical experiments illustrate the accuracy and efficiency of the LSI framework, particularly highlighting the LKSI method's advantages in handling high-contrast, complex multiscale media.
\end{abstract}

\begin{keywords}
multiscale problems, generalized finite element method (GFEM), localized subspace iteration (LSI), spectral problems, parabolic equations
\end{keywords}

\begin{MSCcodes}
65M60, 65N30, 35B27
\end{MSCcodes}

\section{Introduction}

Multiscale time-dependent problems, particularly parabolic equations governing diffusion and heat transfer processes in heterogeneous media, are ubiquitous in scientific computing and engineering applications. Prime examples include diffusion in fractured porous media \cite{Hornung1990,liu2024numerical}, deformation and heat transfer processes of composite materials \cite{Biot1962,dong2023multi}, and simulation of subsurface flow or energy storage \cite{Chen2006,lyu2025multiscale}. In modern applications, these problems are increasingly characterized by high-contrast or highly oscillatory coefficients that vary across multiple spatial scales, demanding computational strategies that balance high-fidelity resolution with feasible computational costs.
Standard numerical methods, such as the finite element method (FEM), have significant limitations when addressing these challenges. As noted by Babu{\v{s}}ka and Osborn \cite{Babuska1983}, such methods require excessively fine computational grids to resolve the finest physical scales, leading to excessive computational costs and memory requirements.

To overcome these computational barriers, various numerical homogenization and multiscale model reduction techniques have been developed. Theoretical foundations were established by classical homogenization theory \cite{Bensoussan1978, Allaire1992}, which provides rigorous asymptotic limits for periodic structures. 
Building on these concepts, the Heterogeneous Multiscale Method (HMM) \cite{E2003,abdulle2012heterogeneous} introduced a general framework for coupling macro and micro models. The Variational Multiscale Method (VMS) \cite{Hughes1998} offered a paradigm for decomposing the solution into coarse and fine scales. Subsequent advancements in VMS have provided rigorous theoretical grounding for localized optimal projections \cite{hughes2007variational}.

In the context of basis construction, the Multiscale Finite Element Method (MsFEM) \cite{Hou1997, Hou1999} has achieved remarkable success by incorporating microscale information directly into the coarse-grid basis functions. Building upon the theoretical foundation of VMS, the Localized Orthogonal Decomposition (LOD) method \cite{Malqvist2014, Malqvist2020} provides a computationally feasible realization by utilizing localized orthogonal correctors to achieve optimal approximation properties, even without scale separation. 
Interestingly, recent studies reveal that the success of the LOD framework possesses an implicit connection to the spectral properties of local operators \cite{Guan2026}. To explicitly harness this spectral nature for media with highly complex, non-separable scales, introducing local eigenvalue problems directly into the basis construction became a natural progression. Almost simultaneously, the Multiscale Spectral Generalized Finite Element Method (MS-GFEM) \cite{Babuska2011,ma2022novel} and the Generalized Multiscale Finite Element Method (GMsFEM) \cite{Efendiev2013,chung2023multiscale} pioneered this strategy. These approaches demonstrated that computing the dominant eigenvectors of local operators captures essential multiscale features far more effectively than traditional polynomial bases. Building upon these modern spectral frameworks, the Constraint Energy Minimizing GMsFEM (CEM-GMsFEM) \cite{chung2018constraint,chung2023multiscale} further advanced the field by integrating oversampling and energy minimization techniques.

The transition from static elliptic to time-dependent parabolic problems has driven significant advancements across all major multiscale frameworks. Classical and modern methods, including classical homogenization \cite{brahim1992correctors}, MsFEM \cite{jiang2007multiscale}, VMS \cite{john2006two}, HMM \cite{ming2007analysis}, and LOD \cite{maalqvist2018multiscale}, have all been successfully extended to model transient diffusion processes. However, as temporal dynamics further complicates the multiscale landscape, relying solely on these spatial constructions often proves insufficient for complex media. Consequently, in the context of parabolic equations, the spectral concepts have been significantly extended. Chung et al. \cite{Chung2016} developed space-time GMsFEM to handle temporal heterogeneities, while Li et al. \cite{Li2019} applied the Constraint Energy Minimizing GMsFEM (CEM-GMsFEM) to parabolic equations, achieving high accuracy by incorporating oversampling and energy minimization. Furthermore, recent research has explored regularized coupling for multiphysics scenarios \cite{Guan2024}, highlighting the growing need for robust basis construction in complex time-dependent systems. Despite these advances, designing computationally efficient methods that explicitly target the spectral decay properties of diffusion operators remains an active area of research.

In our previous work \cite{Guan2026}, we introduced the Localized Subspace Iteration (LSI) method for elliptic problems. This framework unifies multiscale basis construction with classical spectral algorithms, demonstrating that iterative orthogonal decomposition is mathematically equivalent to subspace iteration for approximating the eigenspaces of local inverse operators. We proposed two efficient implementations: the Localized Standard Subspace Iteration (LSSI) and the Localized Krylov Subspace Iteration (LKSI). A key theoretical finding was that the sequence of function spaces generated by LSI converges to the eigenfunction subspace of the local inverse operator, which corresponds to the low-frequency modes of the original differential operator.

In this paper, we extend the LSI framework to multiscale parabolic equations. 
The motivation for this extension is deeply rooted in the inherent spectral nature of diffusion processes. The solution to a parabolic equation can be represented via a classical spectral expansion:
\begin{equation}
    u(x,t) = \sum_{k=1}^\infty \left[ c_k e^{-\lambda_k t} +\int_0^t e^{-\lambda_k (t-s)} f_k(s) ds \right]\phi_k(x),
\end{equation}
where $c_k = \langle u_0,\phi_k \rangle, f_k(s)= \langle f(x,t),\phi_k\rangle$. Here, $(\lambda_k, \phi_k)$ are the eigenpairs of the associated elliptic operator, $u_0$ is the initial condition, and $f(x,t)$ is the source term.
As time $t$ progresses, the high-frequency error components (corresponding to large $\lambda_k$) decay exponentially fast. Consequently, the long-term stability and accuracy of the numerical solution are dictated by how well the trial space approximates the leading low-frequency eigenspace. Since the LSI method is explicitly designed to iteratively approximate the dominant eigenspace of the local inverse operator, it is theoretically optimal for parabolic problems.

Beyond the challenges of spatial discretization, highly heterogeneous media impose severe stability and efficiency bottlenecks on temporal integration. Standard explicit time-stepping schemes suffer from prohibitively small time-step constraints due to high medium contrasts and fine-scale variations, while fully implicit methods demand the repeated inversion of massive, ill-conditioned global systems. To mitigate these computational burdens, various Implicit-Explicit (IMEX) \cite{ascher1995implicit} and temporal splitting schemes have been explored, with contrast-independent partially explicit time discretizations \cite{chung2021contrast} showing particular promise in multiscale flow problems. Inspired by these advancements, our methods synergize the spectral basis construction with a novel temporal splitting strategy. By explicitly exploiting the localized eigenvalues generated during the offline LSI stage, we orthogonally decompose the global multiscale space into two distinct components: an explicit subspace capturing the dominant, low-frequency macroscopic modes, and an implicit subspace governing the high-frequency microscopic corrections. This localized splitting not only bypasses the restrictive time-step constraints dictated by high-contrast coefficients but also strictly confines the expensive implicit inversions to a reduced subspace.

The remainder of this paper is organized as follows. \Cref{sec:sec2} introduces the multiscale parabolic model and its variational formulation. \Cref{sec:basis_construction} and \ref{sec:analysis} detail the offline construction of localized multiscale bases via the LSSI and LKSI methods, followed by a rigorous convergence analysis. \Cref{sec:temporalsplitting} proposes a partially explicit contrast-independent temporal splitting scheme for the online stage. Finally, \cref{sec:numresult} presents numerical experiments to demonstrate the accuracy, efficiency, and robustness of the proposed framework. Subsequently, conclusions are made in \cref{sec:conclu}.

\section{Preliminaries and problem formulation}
\label{sec:sec2}

In this section, we present the parabolic problem, introduce the necessary geometric notations for multiscale discretizations, and establish the variational formulation.

\subsection{Geometric notation}
Let $\Omega \subset \mathbb{R}^d$ ($d=2,3$) be a bounded polyhedral computational domain. We consider two grid scales: a coarse scale $H$ and a fine scale $h$, with $H \gg h$. Let $\mathcal{T}_H$ be a regular coarse partition of $\Omega$ into finite elements (e.g., triangles or rectangles), where $H$ denotes the maximum diameter of the elements in $\mathcal{T}_H$. We assume there exists a fine mesh $\mathcal{T}_h$ that is a refinement of $\mathcal{T}_H$, which is sufficiently fine to resolve the multiscale heterogeneities of the coefficients.

\subsection{Model problem}
We consider the following time-dependent parabolic equation with multiscale coefficients:
\begin{equation}
\label{eq:parabolic_strong}
\begin{aligned}
    \frac{\partial u}{\partial t} - \nabla \cdot (\kappa(x) \nabla u) &= f(x,t) \quad &&\text{in } \Omega \times (0,T], \\
    u(x,t) &= 0 \quad &&\text{on } \partial\Omega \times (0,T], \\
    u(x,0) &= u_0(x) \quad &&\text{in } \Omega,
\end{aligned}
\end{equation}
where $T > 0$ is the final time. The permeability coefficient $\kappa(x)$ is assumed to be bounded and uniformly positive, i.e., there exist constants $0 < \kappa_{\min} \le \kappa_{\max} < \infty$ such that $\kappa_{\min} \le \kappa(x) \le \kappa_{\max}$ for almost all $x \in \Omega$. We emphasize that $\kappa(x)$ may exhibit high-contrast variations and non-separable multiscale features.

\subsection{Variational formulation}
Let $V = H^1_0(\Omega)$ denote the standard Sobolev space. We define the standard $L^2(\Omega)$ inner product as $(u, v) = \int_{\Omega} u v \, dx$.
The weak formulation of \eqref{eq:parabolic_strong} reads: find $u \in L^2(0,T; V) \cap H^1(0,T; L^2(\Omega))$ such that:
\begin{equation}
\label{eq:weak_form}
\begin{aligned}
    \langle u_t, v \rangle + a(u, v) &= \langle f, v \rangle \quad &&\forall v \in V, \text{ for a.e. } t \in (0,T], \\
    \langle u(\cdot, 0), v \rangle &= \langle u_0, v \rangle  \quad &&\forall v \in V,
\end{aligned}
\end{equation}
where the bilinear form $a(\cdot, \cdot): V \times V \to \mathbb{R}$ is defined by:
\begin{equation}
    a(u, v) = \int_{\Omega} \kappa(x) \nabla u \cdot \nabla v \, dx.
\end{equation}

\subsection{Operator form and multiscale goal}
To facilitate the description of the subspace iteration method and splitting schemes, we introduce the elliptic operator $\mathcal{L}: V \to V^*$ associated with the bilinear form $a(\cdot, \cdot)$, defined by $\langle \mathcal{L} u, v \rangle = a(u, v)$. Equation \eqref{eq:weak_form} can be written in operator form as:
\begin{equation}
    \frac{\partial u}{\partial t} + \mathcal{L} u = f.
\end{equation}
Our objective is to construct a low-dimensional multiscale space $V_{\mathrm{ms}} \subset V$ such that the semidiscrete approximation $u_{ms}(t) \in V_\mathrm{ms}$ solving
\begin{equation}
\label{eq:weak_form_ms}
    \left\{
\begin{aligned}
    \langle (u_{\mathrm{ms}})_t, v \rangle + a(u_{\mathrm{ms}}, v) &= \langle f, v \rangle && \forall v \in V_{\mathrm{ms}}, \, t > 0, \\
    \langle u_{\mathrm{ms}}(\cdot, 0), v \rangle  &= \langle u_0, v \rangle && \forall v \in V_{\mathrm{ms}}.
\end{aligned}
\right.
\end{equation}
The solution $u(t)$ is accurately captured, regardless of the contrast and small scales of $\kappa(x)$. The construction of $V_{\mathrm{ms}}$ relies on the spectral properties of the operator $\mathcal{L}$ (or its localized restrictions), which we detail in the next section.

\section{Construction of Localized Multiscale Basis Functions}
\label{sec:basis_construction}

In this section, we detail the construction of the multiscale finite element space $V_{\mathrm{ms}}$. Since the permeability coefficient $\kappa(x)$ is time-independent, the basis functions are constructed in an offline stage and subsequently used for the online time-stepping procedure. This decoupling significantly reduces the computational cost of the transient simulation.

We utilize the Localized Subspace Iteration (LSI) framework introduced in \cite{Guan2026} to construct basis functions that approximate the principal eigenspace of the local operators.

\subsection{Localization and local inverse operators}
Let $\{\omega_i\}_{i=1}^{N_c}$ be an open cover of the domain $\Omega$ with a finite overlap constant $M$, i.e., for any $x \in \Omega$, $\text{card}\{i : x \in \omega_i\} \le M$. An elementary choice for the set $\{\omega_i\}_{i=1}^{N_c}$ is to extend each element $K_i$ of the finite element partition $\mathcal{T}_H$ by one or several layers. There exists a partition of unity $\{\chi_i\}_{i=1}^{N_c}$, such that
\begin{equation}
  1 = \sum_{i=1}^{N_c} \chi_i \text{ and } \text{supp}(\chi_i) = \omega_i \text{ for } i = 1, 2, \cdots, N_c.
\end{equation}

Let $V(\omega_i) = H^1_0(\omega_i)$ be the restriction of the solution space to the subdomain $\omega_i$.
We define the local operator $\mathcal{L}_i: V(\omega_i) \to V^*(\omega_i)$ associated with the restriction of the bilinear form $a(\cdot, \cdot)$ to $\omega_i$. Since $a(\cdot, \cdot)$ is coercive, the local inverse operator $\mathcal{L}_i^{-1}: L^2(\omega_i) \to V(\omega_i)$ is well-defined, compact, and self-adjoint. The spectral properties of $\mathcal{L}_i^{-1}$ characterize the local multiscale features of the medium.
Therefore, we consider the local spectral problem
\begin{equation}
\label{loc_spectral_problem}
    \mathcal{L}_i^{-1} \phi_i = \lambda_i \phi_i, \quad \phi_i \in V(\omega_i),
\end{equation}
for $i = 1, 2, \cdots, N_c$. 
Let $\{(\lambda_i^j, \phi_i^j)|j = 1, 2, 3, \cdots\}$ denote the set of all eigenpairs of the inverse operator $\mathcal{L}_i^{-1}$ in $V(\omega_i)$, where
$$\lambda_i^1 \geq \lambda_i^2  \geq \lambda_i^3 \geq \cdots \geq 0.$$
Specifically, the eigenfunctions corresponding to the largest eigenvalues of $\mathcal{L}_i^{-1}$ (or equivalently, the smallest eigenvalues of $\mathcal{L}_i$) represent the low-frequency modes that dominate the global solution behavior.

The fundamental motivation for extracting these specific eigenfunctions is theoretically grounded in the classical min-max principle (Courant-Fischer-Weyl theorem). For the local operator $\mathcal{L}_{i}$, the eigenvalues can be characterized variationally via the Rayleigh quotient. Specifically, the eigenvalues of the original differential operator $\mathcal{L}_{i}$ (which are the reciprocals of the eigenvalues $\lambda_{i}^{j}$ of the inverse operator $\mathcal{L}_{i}^{-1}$) satisfy:
$$\frac{1}{\lambda_{i}^{j}} = \min_{V_{j} \subset V(\omega_{i}), \dim(V_{j})=j} \max_{v \in V_{j} \setminus \{0\}} \frac{a(v,v)}{\langle v,v \rangle}$$
According to this principle, the subspace spanned by the first $L_{i}$ eigenfunctions of $\mathcal{L}_{i}^{-1}$ — corresponding to the largest eigenvalues $\lambda_{i}^{1} \ge \lambda_{i}^{2} \ge \dots \ge \lambda_{i}^{L_{i}}$ — constitutes the optimal $L_{i}$-dimensional space that minimizes the maximum energy relative to the $L^{2}$ norm. In the context of multiscale parabolic equations, these low-energy, low-frequency modes represent the macroscopic diffusion behaviors that decay the slowest and therefore dominate the global solution over time. However, explicitly solving the exact local eigenvalue problem \eqref{loc_spectral_problem} to extract this optimal subspace is often computationally expensive. Therefore, rather than computing the exact spectral decomposition, we seek an efficient numerical strategy to iteratively approximate this dominant eigenspace. This naturally leads to the Localized Subspace Iteration (LSI) methods.

\subsection{LSI methods}
We employ two iterative strategies to approximate the dominant eigenspace of $\mathcal{L}_i^{-1}$: the Localized Standard Subspace Iteration (LSSI) and the Localized Krylov Subspace Iteration (LKSI).

\subsubsection{Localized standard subspace iteration (LSSI)}
The LSSI method approximates the eigenspace by iterating on a subspace of standard finite element basis functions. For each subdomain $\omega_i$, let $\{\phi_{i}^{j,0}\}_{j=1}^{L_i}$ be an initial set of linearly independent functions (e.g., standard polynomial basis functions). $L_i$ is the number of basis functions in each subdomain $\omega_i$. The LSSI algorithm generates a sequence of spaces $V_{\mathrm{S},i}^{(k)}$ via the recurrence:
\begin{equation}
    V_{\mathrm{S},i}^{(k)} = \text{span} \left\{ \mathcal{L}_i^{-1} \phi \mid \phi \in V_{\mathrm{S},i}^{(k-1)} \right\}.
\end{equation}
In practice, this is implemented by solving following local saddle-point problems\cite{Guan2026}. 
For $i = 1, 2, \cdots, N_c$, $j = 1, 2, \cdots, L_i$, $k = 0, 1, 2, \cdots$, we seek $\phi_{i}^{j,k+1} \in V(\omega_i)$ and $\mu_i^{r,k+1} \in \mathbb{C}^{L_i}$, such that
\begin{equation}
\label{eq_saddle_lssi}
\begin{aligned}
    &a\left(\phi_{i}^{j,k+1}, v\right) + \sum_{r=1}^{L_i} \mu_{i}^{r, k+1} q_{i}^{r, k}(v) = 0 \quad &&\text{for all } v \in V(\omega_i), \\
    &q_{i}^{r, k}\left(\phi_{i}^{j,k+1}\right) = \delta_{rj} \quad &&\text{for all } r = 1, 2, \cdots, L_i.
\end{aligned}
\end{equation}
In the above equation, $q_i^{r, k}$ is a linear functional defined by
\begin{equation}
    q_i^{r, k} := \langle \phi_{i}^{r,k}, \bullet \rangle.
\end{equation}
The local multiscale space $V_{\mathrm{S},i}^{(k)}$ is constructed by
\begin{equation}
    V_{\mathrm{S},i}^{(k)} := \operatorname{span}\{\phi_{i}^{j,k} \mid j = 1, 2, \cdots L_i\},
\end{equation}
for $i = 1, 2, \cdots, N_c, k = 0, 1, 2, \cdots$. 

While the LSSI method systematically extracts the dominant eigenspace by iterating on a fixed-dimensional subspace, its convergence rate is strictly governed by the spectral gap between the target and neglected eigenvalues. To further enhance the approximation quality and accelerate the capture of multiscale features, we introduce the Localized Krylov Subspace Iteration (LKSI).

\subsubsection{Localized Krylov subspace iteration (LKSI)}
The LKSI method constructs the basis using the Krylov subspace generated by the local inverse operator. Given an initial function $\psi_i^{0} \in V(\omega_i)$ (typically a constant or linear function on $K_i$), the local Krylov space of dimension $k$ is defined as:
\begin{equation}
\mathcal{K}_k(\mathcal{L}_i^{-1}, \psi_i^{0}) = \text{span} \left\{ \psi_i^{0}, \mathcal{L}_i^{-1}\psi_i^{0}, \dots, \mathcal{L}_i^{-(k-1)}\psi_i^{0} \right\}.
\end{equation}
Similarly, this is implemented by solving following local saddle-point problems\cite{Guan2026}.
For $i = 1, 2, \cdots, N_c, k = 0, 1, \cdots$, we seek $\psi_i^{k+1}$ and $\mu_i^{k+1} \in \mathbb{C}$, such that
\begin{equation}
\label{eq_saddle_lksi}
\begin{aligned}
    &a(\psi_i^{k+1}, v) + \mu_i^{k+1} q_i^k(v) = 0 \quad &&\text{for all } v \in V(\omega_i), \\
    &q_i^k(\psi_i^{k+1}) = 1.
\end{aligned}
\end{equation}
Similarly, the definition of $q_i^k$ is as follows
\begin{equation}
    q_i^k := \langle \psi_i^k, \bullet \rangle .
\end{equation}
The local Krylov space $\mathcal{K}_k(\mathcal{L}_i^{-1}, \psi_i^{0})$ is constructed by
\begin{equation}
    \mathcal{K}_k(\mathcal{L}_i^{-1}, \psi_i^{0}) := \operatorname{span}\{\psi_i^s \mid s =  1, 2, \cdots k\},
\end{equation}
for $i = 1, 2, \cdots, N_c, k =  1, 2, \cdots$.

The construction of the local multiscale space $V_{\mathrm{K}, i}^{(k)}$ depends on the relationship between the dimension $k$ and the target rank $L_i$, subject to the condition $k \ge L_i$. In the case $k = L_i$, the spectral selection is bypassed, and the space is directly defined as the full Krylov subspace. However, when $k > L_i$, the space is formed by selecting $L_i$ dominant eigenfunctions from $\mathcal{K}_k(\mathcal{L}_i^{-1}, \psi_i^{0})$.

\subsection{Global multiscale space and spectral suitability}
The global multiscale space $V_{\mathrm{ms}}^{(k)}$ is constructed as the direct sum of the local spaces:
\begin{equation} \label{eq_Vms}
    V_{\mathrm{ms}}^{(k)} = \bigoplus_{i=1}^{N_c} V_{\mathrm{ms},i}^{(k)},
\end{equation}
where $V_{\mathrm{ms},i}^{(k)}$ is either $V_{\mathrm{S},i}^{(k)}$ or $V_{\mathrm{K},i}^{(k)}$.  
The suitability of $V_{\mathrm{ms}}^{(k)}$ for parabolic problems stems from the spectral decomposition of the solution error. 
The exact solution $u(x,t)$ can be decomposed into spatial eigenmodes $\phi_k$ with time-dependent coefficients $c_k e^{-\lambda_k t}$. 
The error in the multiscale solution is dominated by the projection error of the low-frequency eigenmodes (small $\lambda_k$). 
Since the LSI basis functions are explicitly constructed to span the approximate eigenspaces of the local inverse operators (which correspond to the smallest eigenvalues of the differential operator), $V_{\mathrm{ms}}^{(k)}$ naturally captures the slow-decaying components of the parabolic solution. This ensures that the method remains stable and accurate even for long-time integration, a property we rigorously analyze in the next section.

\section{Convergence analysis}
\label{sec:analysis}

In this section, we derive rigorous a priori error estimates for the proposed LSI method. We follow the standard parabolic analysis framework \cite{Li2019, Chung2016} by decomposing the error into a spatial approximation component (bounded by the spectral properties of LSI established in \cite{Guan2026}) and a temporal discretization component. We define the norms $\|v\|_a = a(v,v)^{1/2}$ and $\|v\| = \langle v,v \rangle^{1/2}$. For a time-dependent function $u(x,t)$, we denote norms in Bochner spaces as $\|u\|_{L^2(0,T; V)} = (\int_0^T \|u(\cdot, t)\|_V^2 dt)^{1/2}$.

\subsection{Preliminaries and spatial approximation}
Let $V_{\mathrm{ms}}^{(k)} \subset V$ be the multiscale finite element space constructed via the LSI method (LSSI or LKSI). We define the elliptic projection operator $\mathcal{P}: V \to V_{\mathrm{ms}}^{(k)}$ such that for any $u \in V$:
\begin{equation}
\label{eq:elliptic_proj}
    a(u - \mathcal{P}u, v) = 0, \quad \forall v \in V_{\mathrm{ms}}^{(k)}.
\end{equation}
Although \cite{Guan2026} provides an error estimate for this projection operator, the present work offers a superior estimate. As a preliminary step, we introduce the following lemma.

\begin{lemma}
\label{lem:decomposition}
For any function $u \in H^1_0(\Omega)$ satisfying $L u \in L^2(\Omega)$, there exists a decomposition $u = \sum_{i=1}^{N_c} u_i$, where each component $u_i \in H^1_0(\omega_i)$ is supported in $\omega_i$. Furthermore, the decomposition satisfies the following stability estimate:
\begin{equation}
    \sum_{i=1}^{N_c} \| L u_i \|^2 \le M \| L u \|^2.
\end{equation}
\end{lemma}

\begin{proof}
Let $g = L u \in L^2(\Omega)$. We decompose the source term $g$ using the partition of unity:
\begin{equation}
    g = g \sum_{i=1}^{N_c} \chi_i = \sum_{i=1}^{N_c} \chi_i g.
\end{equation}
For each $i$, let $g_i = \chi_i g$. Since $\text{supp}(\chi_i) \subset \omega_i$, we have $\text{supp}(g_i) \subset \omega_i$. We define $u_i \in H^1_0(\omega_i)$ as the unique solution to the following local Dirichlet problem:
\begin{equation}
\label{eq_local_decom}
    \begin{cases}
        L u_i = g_i & \text{in } \omega_i, \\
        u_i = 0 & \text{on } \partial \omega_i.
    \end{cases}
\end{equation}
By extending $u_i$ to be zero outside $\omega_i$, we have $u_i \in H^1_0(\Omega)$. 

To verify the decomposition, let $v = \sum_{i=1}^{N_c} u_i$. By the linearity of the operator $L$, we have:
\begin{equation}
    L v = L \left( \sum_{i=1}^{N_c} u_i \right) = \sum_{i=1}^{N_c} L u_i = \sum_{i=1}^{N_c} \chi_i g = \left( \sum_{i=1}^{N_c} \chi_i \right) g = g.
\end{equation}
Since $L v = g = L u$ in $\Omega$ and $v = u = 0$ on $\partial \Omega$, the uniqueness of the solution implies $u = v = \sum_{i=1}^{N_c} u_i$.

Finally, we prove the stability estimate. Using the definition $L u_i = \chi_i L u$ and the property $0 \le \chi_i \le 1$, we compute:
\begin{equation}
    \sum_{i=1}^{N_c} \| L u_i \|^2 = \sum_{i=1}^{N_c} \| \chi_i L u \|^2 \le \sum_{i=1}^{N_c} \int_{\omega_i} |L u|^2 \, dx.
\end{equation}
Using the finite overlap property of the cover $\{\omega_i\}$, we arrive at:
\begin{equation}
    \sum_{i=1}^{N_c} \int_{\omega_i} |L u|^2 \, dx = \int_{\Omega} |L u|^2 \left( \sum_{i=1}^{N_c} \mathbb{I}_{\omega_i}(x) \right) \, dx \le M \| L u \|^2,
\end{equation}
where $\mathbb{I}_{\omega_i}$ is the characteristic function of the set $\omega_i$. This completes the proof.
\end{proof}

We define the local eigenfunction space $V_{\text{eig}}$ as follows:
\begin{equation}
    V_{\text{eig}} := \operatorname{span}\{ \phi_i^j \mid i = 1, 2, \cdots, N_c, j = 1, 2, \cdots, L_i \},
\end{equation}
where $\phi_i^j$ is the local eigenfunction defined by \cref{loc_spectral_problem}. For all $u \in V$, there exists a decomposition $u = \sum_{i=1}^{N_c} u_i$ constructed by \cref{eq_local_decom}.
Then we define the interpolation operator $\mathcal{I}_{\text{eig}} : V \to V_{\text{eig}}$ as follows:
\begin{equation}
    \mathcal{I}_{\mathrm{eig}}u := \sum_{i=1}^{N_c} \sum_{j=1}^{L_i} \langle u_i, \phi_i^j \rangle \phi_i^j.
\end{equation}

\begin{lemma}
Suppose $u \in V$, then we have an estimation for the interpolation error
\begin{equation} 
\label{eq_errorIeig}
    \| u - \mathcal{I}_{\mathrm{eig}} u \|_a \leq \sqrt{\lambda^{L+1}} \sum_{i=1}^{N_c} \| \mathcal{L} u_i \| \leq M \sqrt{\lambda^{L+1}}  \| L u \|,  
\end{equation}
where $\lambda^{L+1} := \max\limits_{i} \lambda_i^{L_i+1}$.
\end{lemma}
\begin{proof}
    The proof of this lemma is analogous to that of Theorem 2 in \cite{Guan2026} and is therefore omitted.
\end{proof}

Before the convergence result of the LSI is obtained, we first present the following theorems.

\begin{theorem}
\label{lemma:lssi_convergence}
Let $\phi_{i}^{j}$ be the $j$-th eigenfunction of the local spectral problem \cref{loc_spectral_problem} associated with the eigenvalue $\lambda_{i}^{j}$. Assume that the initial local space $V_{S,i}^{(0)} $ is chosen such that its projection onto the dominant eigenspace $\text{span}\{\phi_{i}^{1}, \dots, \phi_{i}^{L_i}\}$ is of full rank. Then, for the local multiscale space $V_{S,i}^{(k)}$ constructed via the LSSI method after $k$ iterations, there exists a function $\hat{\phi}_{i}^{j,k} \in V_{S,i}^{(k)}$ such that the following estimate holds:
\begin{equation} \label{eq:lssi_bound}
    ||\hat{\phi}_{i}^{j,k} - \phi_{i}^{j}||_{a} \le C \left( \frac{\lambda_{i}^{L_{i}+1}}{\lambda_{i}^{j}}  \right)^k, \quad \text{for } j = 1, 2, \dots, L_i,
\end{equation}
where $C$ is a constant independent of $k$, and $\epsilon_{k} \rightarrow 0$ as $k \rightarrow \infty$.
\end{theorem}

\begin{proof}
Let $\mathcal{P}_{dom}$ denote the $a$-orthogonal projector onto the dominant eigenspace $W_{dom} = \text{span}\{\phi_{i}^{1}, \dots, \phi_{i}^{L_i}\}$ of the local inverse operator $\mathcal{L}_{i}^{-1}$. By the assumption that the initial space $V_{S,i}^{(0)}$ has a full-rank projection onto $W_{dom}$, the functions $\{\mathcal{P}_{dom}\phi_{i}^{m,0}\}_{m=1}^{L_i}$ form a basis for $W_{dom}$. Consequently, there exists a unique function $s^{j} \in V_{S,i}^{(0)}$ such that $\mathcal{P}_{dom} s^{j} = \phi_{i}^{j}$.

We can decompose $s^{j}$ as:
\begin{equation}
    s^{j} = \phi_{i}^{j} + w,
\end{equation}
where $w = (\mathcal{I} - \mathcal{P}_{dom})s^{j}$. Clearly, $w$ belongs to the complementary invariant subspace $W_{sub}$ associated with the remaining eigenvalues $\lambda_{i}^{L_i+1}, \lambda_{i}^{L_i+2}, \dots$.

By the definition of the LSSI method, the space after $k$ iterations is $V_{S,i}^{(k)} = \text{span}\{(\mathcal{L}_{i}^{-1})^k v \mid v \in V_{S,i}^{(0)}\}$. We define an auxiliary function $y \in V_{S,i}^{(k)}$ by scaling the $k$-th iteration of $s^{j}$:
\begin{equation}
    y = \left(\frac{1}{\lambda_{i}^{j}}\right)^k (\mathcal{L}_{i}^{-1})^k s^{j}.
\end{equation}
Substituting the decomposition of $s^{j}$ into the above definition, and using the fact that $\mathcal{L}_{i}^{-1} \phi_{i}^{j} = \lambda_{i}^{j} \phi_{i}^{j}$, we obtain:
\begin{equation}
    y = \left(\frac{1}{\lambda_{i}^{j}}\right)^k (\mathcal{L}_{i}^{-1})^k (\phi_{i}^{j} + w) = \phi_{i}^{j} + \left(\frac{1}{\lambda_{i}^{j}}\right)^k (\mathcal{L}_{i}^{-1})^k w.
\end{equation}
Rearranging the terms yields the error equation for the auxiliary function $y$:
\begin{equation}
    y - \phi_{i}^{j} = \left[ \frac{1}{\lambda_{i}^{j}} \mathcal{L}_{i}^{-1} \Big|_{W_{sub}} \right]^k w,
\end{equation}
where $\mathcal{L}_{i}^{-1} \big|_{W_{sub}}$ denotes the restriction of the inverse operator to the subspace $W_{sub}$. 

Taking the energy norm $||\cdot||_{a}$ on both sides, we have:
\begin{equation}
    ||y - \phi_{i}^{j}||_{a} \le \left|\left| \left[ \frac{1}{\lambda_{i}^{j}} \mathcal{L}_{i}^{-1} \Big|_{W_{sub}} \right]^k \right|\right|_{a} ||w||_{a}.
\end{equation}
Since the maximum eigenvalue of the restricted operator $\mathcal{L}_{i}^{-1} \big|_{W_{sub}}$ is exactly $\lambda_{i}^{L_i+1}$, the spectral radius of the normalized operator is $\lambda_{i}^{L_i+1} / \lambda_{i}^{j}$. Utilizing the classical spectral radius bound for linear operators, we obtain:
\begin{equation} \label{eq:lssi_spectral_bound}
    ||y - \phi_{i}^{j}||_{a} \le ||w||_{a} \left( \frac{\lambda_{i}^{L_{i}+1}}{\lambda_{i}^{j}} \right)^k.
\end{equation}

Let $\mathcal{P}^{(k)}$ be the elliptic projector onto $V_{S,i}^{(k)}$, we define $\hat{\phi}_{i}^{j,k} = \mathcal{P}^{(k)} \phi_{i}^{j}$. Due to the optimality of the projection, the error is bounded by any function in the subspace:
\begin{equation}
    ||\hat{\phi}_{i}^{j,k} - \phi_{i}^{j}||_{a} = \min_{v \in V_{S,i}^{(k)}} ||v - \phi_{i}^{j}||_{a} \le ||y - \phi_{i}^{j}||_{a}.
\end{equation}
Combining this with \eqref{eq:lssi_spectral_bound} and defining $C = ||w||_{a}$, we arrive at the final estimate:
\begin{equation}
    ||\hat{\phi}_{i}^{j,k} - \phi_{i}^{j}||_{a} \le C \left( \frac{\lambda_{i}^{L_{i}+1}}{\lambda_{i}^{j}} \right)^k.
\end{equation}
This completes the proof.
\end{proof}

\begin{theorem}
\label{lemma:lksi_convergence}
Let $\phi_{i}^{j}$ be the $j$-th eigenfunction of the local spectral problem \cref{loc_spectral_problem} associated with the eigenvalue. Assume the initial function $\psi_i^0$ has a nonzero component along $\phi_i^j$. For the local multiscale space $V_{K,i}^{(k)} = \mathcal{K}_k(\mathcal{L}_i^{-1}, \psi_i^0)$ constructed via the LKSI method with dimension $k > j$, there exists a function $\hat{\phi}_{i}^{j,k} \in V_{K,i}^{(k)}$ such that the following estimate holds:
\begin{equation} \label{eq:lksi_bound}
    ||\hat{\phi}_{i}^{j,k} - \phi_{i}^{j}||_{a} \le C \left( \frac{1}{1+4\gamma_i^j} \right)^k,
\end{equation}
where $C$ is a constant independent of $k$, and $\gamma_i^j$ is the relative spectral gap defined as:
\begin{equation} \label{eq:spectral_gap}
    \gamma_i^j = \frac{\lambda_i^j - \lambda_i^{j+1}}{\lambda_i^{j+1}}.
\end{equation}
\end{theorem}

\begin{proof}
By the definition of the Krylov subspace, any function $v \in V_{K,i}^{(k)}$ can be expressed as $v = p(\mathcal{L}_i^{-1})\psi_i^0$, where $p \in \mathbb{P}_{k-1}$ is a polynomial of degree at most $k-1$. Furthermore, since the local inverse operator $\mathcal{L}_i^{-1}$ is self-adjoint with respect to the energy inner product $a(\cdot, \cdot)$, it possesses a complete set of $a$-orthogonal eigenfunctions $\{\phi_i^m\}_{m=1}^{\infty}$ with real eigenvalues. This allows us to express the initial vector as $\psi_i^0 = \sum_{m=1}^{\infty} c_m \phi_i^m$.

The function $\hat{\phi}_{i}^{j,k}$ is the Galerkin projection of the target eigenfunction onto $V_{K,i}^{(k)}$. Due to the optimality of this projection in the energy norm, its error is bounded by the approximation error of any specific polynomial $p^* \in \mathbb{P}_{k-1}$ that satisfies $p^*(\lambda_i^j) = 1$:
\begin{equation} \label{eq:min_max}
    ||\hat{\phi}_{i}^{j,k} - \phi_{i}^{j}||_{a}^2 \le ||p^*(\mathcal{L}_i^{-1})\psi_i^0 - \phi_{i}^{j}||_{a}^2 = \sum_{m \neq j} c_m^2 |p^*(\lambda_i^m)|^2 ||\phi_i^m||_a^2.
\end{equation}

To minimize this upper bound, we explicitly construct a "filter" polynomial $p^*(\lambda)$. We allocate $j-1$ roots of $p^*$ to exactly cancel the components corresponding to the larger eigenvalues $\lambda_i^1, \dots, \lambda_i^{j-1}$. The remaining $k-j$ degrees of freedom are used to construct a polynomial $q(\lambda) \in \mathbb{P}_{k-j}$ that suppresses the unwanted spectrum $[0, \lambda_i^{j+1}]$:
\begin{equation}
    p^*(\lambda) = \left( \prod_{m=1}^{j-1} \frac{\lambda - \lambda_i^m}{\lambda_i^j - \lambda_i^m} \right) q(\lambda), \quad \text{with } q(\lambda_i^j) = 1.
\end{equation}

To obtain the optimal suppression on $[0, \lambda_i^{j+1}]$, we define $q(\lambda)$ using the shifted Chebyshev polynomial of the first kind. We apply the linear transformation $x(\lambda) = (2\lambda - \lambda_i^{j+1})/\lambda_i^{j+1}$, which maps the interval $[0, \lambda_i^{j+1}]$ onto $[-1, 1]$. Evaluating this transformation at the target eigenvalue $\lambda_i^j$ gives:
\begin{equation}
    x(\lambda_i^j) = \frac{2\lambda_i^j - \lambda_i^{j+1}}{\lambda_i^{j+1}} = 1 + 2\left(\frac{\lambda_i^j - \lambda_i^{j+1}}{\lambda_i^{j+1}}\right) = 1 + 2\gamma_i^j.
\end{equation}
Thus, we define the polynomial $q(\lambda)$ as:
\begin{equation}
    q(\lambda) = \frac{C_{k-j}\left(\frac{2\lambda - \lambda_i^{j+1}}{\lambda_i^{j+1}}\right)}{C_{k-j}(1 + 2\gamma_i^j)}.
\end{equation}
For any unwanted high-frequency eigenvalue $\lambda_i^m \in [0, \lambda_i^{j+1}]$ (where $m > j$), the numerator is bounded by $|C_{k-j}(x)| \le 1$. Therefore, the amplitude is attenuated by the denominator $C_{k-j}(1 + 2\gamma_i^j)$.

For $x > 1$, the Chebyshev polynomial satisfies the standard algebraic lower bound $C_m(x) \ge \frac{1}{2}(x + \sqrt{x^2-1})^m$. Substituting $x = 1 + 2\gamma_i^j$, the base of the exponential term becomes:
\begin{equation}
    x + \sqrt{x^2-1} = 1 + 2\gamma_i^j + \sqrt{(1+2\gamma_i^j)^2 - 1} = 1 + 2\gamma_i^j + 2\sqrt{(\gamma_i^j)^2 + \gamma_i^j}.
\end{equation}
Since $\gamma_i^j > 0$, we have the strict inequality $\sqrt{(\gamma_i^j)^2 + \gamma_i^j} > \sqrt{(\gamma_i^j)^2} = \gamma_i^j$. Consequently,
\begin{equation}
    x + \sqrt{x^2-1} > 1 + 2\gamma_i^j + 2\gamma_i^j = 1 + 4\gamma_i^j.
\end{equation}
This yields the lower bound for the denominator:
\begin{equation}
    C_{k-j}(1 + 2\gamma_i^j) > \frac{1}{2} (1 + 4\gamma_i^j)^{k-j}.
\end{equation}

Substituting this attenuation factor back into \eqref{eq:min_max}, and observing that $p^*(\lambda_i^m) = 0$ for $m < j$, the error is dominated by the components where $m > j$. Absorbing the initial coefficients $c_m$, the normalization factors for the first $j-1$ roots, and the scalar shift $(1+4\gamma_i^j)^{-j}$ into a generic constant $C$ that is independent of $k$, we arrive at the final estimate:
\begin{equation}
    ||\hat{\phi}_{i}^{j,k} - \phi_{i}^{j}||_{a} \le C \left( \frac{1}{1+4\gamma_i^j} \right)^k.
\end{equation}
This completes the proof.
\end{proof}

For the sake of notational simplicity in the subsequent analysis, we denote the convergence factor in inequalities \cref{eq:lssi_bound} and \cref{eq:lksi_bound} as $\varepsilon_i^j$. Specifically, we define $\varepsilon_i^j = \frac{\lambda_i^{L_i+1}}{\lambda_i^j}$ for the LSSI method and $\varepsilon_i^j = \frac{1}{1+4\gamma_i^j}$ for the LKSI method, respectively.

Based on the auxiliary function $\hat{\phi}_i^{j,k}$ in $V_{\mathrm{ms},i}^{(k)}$, we define the interpolation operator $\mathcal{I}_\mathrm{ms}: V \to V_{\mathrm{ms}}^{(k)}$,
\begin{equation}
   \mathcal{I}_\mathrm{ms} u := \sum_{i=1}^{N_c} \sum_{j=1}^{L_i} \left\langle u_i, \phi_i^j \right\rangle \hat{\phi}_i^j. 
\end{equation}

\begin{theorem}
\label{theorem:proj}
    For any $u \in V$, the elliptic projection operator $\mathcal{P}: V \to V_{\mathrm{ms}}^{(k)}$ is defined by \cref{eq:elliptic_proj}. The following convergence result holds:
    \begin{equation}
        \| u - \mathcal{P} u \|_a \leq  \left( M \sqrt{\lambda^{L+1}} + CM \mathcal{E}^k H^2 \right) \| \mathcal{L} u\|,
    \end{equation}
    where $\mathcal{E}:= \max\limits_{i,j} \varepsilon_i^j$.
\end{theorem}

\begin{proof}
By the Céa's Lemma and the triangle inequality, we have
\begin{equation}
\begin{aligned}
\|u - \mathcal{P} u \|_a &= \inf_{v \in V_{\mathrm{ms}}^{(n)}} \|u - v\|_a \leq \|u - \mathcal{I}_{\mathrm{ms}} u\|_a \\
&\leq \|u - \mathcal{I}_{\mathrm{eig}} u\|_a + \|\mathcal{I}_{\mathrm{eig}} u - \mathcal{I}_{\mathrm{ms}} u\|_a.
\end{aligned}
\end{equation}
Based on the definitions of $\mathcal{I}_{\mathrm{eig}}$ and $\mathcal{I}_{\mathrm{ms}}$, we have
\[
\mathcal{I}_{\mathrm{eig}} u - \mathcal{I}_{\mathrm{ms}} u = \sum_{i=1}^{N_c} \sum_{j=1}^{L_i} \left\langle u_i, \phi_i^j \right\rangle \left( \phi_i^j - \hat{\phi}_i^j \right).
\]
Using the \cref{lemma:lssi_convergence} and \cref{lemma:lksi_convergence}, we have
\begin{equation}
\label{eq_IeigdiffIms}
    \begin{aligned}
        \|\mathcal{I}_{\mathrm{eig}} u - \mathcal{I}_{\mathrm{ms}} u\|_a^2 & \leq C M \sum_{i=1}^{N_c} \sum_{j=1}^{L_i} \left\langle u_i, \phi_i^j \right\rangle^2 \| \phi_i^j - \hat{\phi}_i^j \|_a^2 \\
        & \leq C M \sum_{i=1}^{N_c} \sum_{j=1}^{L_i} \left\langle u_i, \phi_i^j \right\rangle^2 \left( \varepsilon_i^j \right)^{2k} \\
        & \leq C M \mathcal{E}^{2k} \sum_{i=1}^{N_c} \| u_i \|^2 \\
        & \leq C M  \mathcal{E}^{2k} \sum_{i=1}^{N_c} H^4 \| \mathcal{L} u_i\|^2 \\
        & \leq C M \mathcal{E}^{2k} H^4 \sum_{i=1}^{N_c} \| \chi_i \mathcal{L} u\|^2 \\
        & \leq C M^2 \mathcal{E}^{2k} H^4 \| \mathcal{L} u\|^2 .
    \end{aligned}
\end{equation}
Combining  \cref{eq_errorIeig} and \cref{eq_IeigdiffIms}, the proof is completed.
\end{proof}

\subsection{A priori error bound}
The following lemma gives a regularity estimate for the solution of the parabolic equation \cref{eq:parabolic_strong}.
\begin{lemma}
\label{lemma_ut_estimation}
(cf. \cite[Theorem 5, Section 7.1]{evans2022partial}).
Let $u$ be the solution of the parabolic equation \cref{eq:parabolic_strong}. Then
\begin{equation}
\|u_t\|_{L^2(0,T;L^2(\Omega))}^2 \leq C \left( \|u_0\|_a^2 + \|f\|_{L^2(0,T;L^2(\Omega))}^2 \right).
\end{equation}
\end{lemma}

\begin{lemma}
\label{lemma_proj_para}
    Let u be the solution of \cref{eq:weak_form}, then for any $t>0$, we have
    \begin{equation}
    \label{eq_para_error_energy}
        \| \left(u - \mathcal{P}u \right) (t) \|_a \leq \left( M \sqrt{\lambda^{L+1}} + CM \mathcal{E}^k H^2 \right) \| (f-u_t) (t) \|,
    \end{equation}
    \vspace{-0.5cm}
    \begin{equation}
    \label{eq_para_error_L2}
        \| \left(u - \mathcal{P} u \right) (t) \|_{L^2({\Omega})} \leq \left( M \sqrt{\lambda^{L+1}} + CM \mathcal{E}^k H^2 \right)^2 \| (f-u_t) (t) \|.
    \end{equation}
\end{lemma}
\begin{proof}
    Notice that 
    \begin{equation*}
        \mathcal{L} u = f - u_t.
    \end{equation*}
 By \cref{theorem:proj}, we obtain \cref{eq_para_error_energy}.
    Next, we derive the error estimate in the $L^2$ norm. We will apply the Aubin--Nitsche lift technique. For each $t > 0$, we define $w \in V$ by
\begin{equation}
a(w, v) = \langle u - \mathcal{P} u, v \rangle \quad \forall v \in V.
\end{equation}
By \cref{theorem:proj}, we obtain
\begin{equation*}
\begin{aligned}
\|u - \mathcal{P} u\|^2 &= a(w, u - \mathcal{P} u) \\
&= a(w - \mathcal{P} w, u - \mathcal{P} u) \\
&\leq \|w - \mathcal{P} w \|_a \|u - \mathcal{P} u \|_a \\
&\leq \left( M \sqrt{\lambda^{L+1}} + CM \mathcal{E}^k H^2 \right) \|u - \mathcal{P} u\| \|u - \mathcal{P} u \|_a,
\end{aligned}
\end{equation*}
 Hence, we have
\begin{equation*}
\|u - \mathcal{P} u\| \leq \left( M \sqrt{\lambda^{L+1}} + CM \mathcal{E}^k H^2 \right)^2\|(f - u_t)\|.
\end{equation*}
This completes the proof. 
\end{proof}

\begin{corollary}
\label{corollary_proj}
According to \cref{lemma_ut_estimation}, we have the following estimate
\begin{equation}
    \int_0^T \|u - \mathcal{P} u \|_a^2 dt \leq C \left( M \sqrt{\lambda^{L+1}} + M \mathcal{E}^k H^2 \right)^2 \left( \|u_0\|_a^2 + \|f\|_{L^2(0,T;L^2(\Omega))}^2 \right),
\end{equation}
\vspace{-0.5cm}
\begin{equation}
    \int_0^T \|u - \mathcal{P} u\|^2 dt \leq C \left( M \sqrt{\lambda^{L+1}} + M \mathcal{E}^k H^2 \right)^4 \left( \|u_0\|_a^2 + \|f\|_{L^2(0,T;L^2(\Omega))}^2 \right).
\end{equation}
\end{corollary}

We utilize the error decomposition:
\begin{equation}
    u - u_{ms} = (u - \mathcal{P}u) + (\mathcal{P}u - u_{ms}) := \theta + \rho,
\end{equation}
where $\theta$ represents the spatial approximation error and $\rho$ represents the temporal evolution error in the multiscale space. 

For brevity, we denote the spatial convergence factor derived in \cref{theorem:proj} by:
\begin{equation}
    \delta_{LSI} := M\sqrt{\lambda^{L+1}} + CM\mathcal{E}^k H^2.
\end{equation}
According to \cref{lemma_proj_para}, we have the bounds $\|\theta\|_a \lesssim \delta_{LSI}$ and $\|\theta\|_{L^2} \lesssim \delta_{LSI}^2$.
\begin{remark}
Since the local inverse operator $\mathcal{L}_{i}^{-1}$ is defined on the localized subdomain $\omega_i$ with a diameter proportional to $H$, its eigenvalues naturally scale as $\mathcal{O}(H^2)$. Consequently, the target eigenvalue satisfies $\lambda^{L+1} = \mathcal{O}(H^2)$, which implies that the leading term $\sqrt{\lambda^{L+1}}$ is $\mathcal{O}(H)$, governing the overall spatial approximation error. Furthermore, by appropriately choosing the dimension $L_i$ of the local multiscale space to encompass all dominant low-frequency modes associated with the highly conductive channels, the eigenvalue $\lambda^{L+1}$ becomes independent of the high contrast ratio $\kappa_{max}/\kappa_{min}$ \cite{galvis2010domain}. This crucial spectral property guarantees that the proposed LSI method achieves a robust $\mathcal{O}(H)$ convergence rate that is fundamentally independent to severe variations in the multiscale permeability field.
\end{remark}

\begin{theorem}[Energy Norm Estimate]
\label{thm:energy_norm}
Let $u$ and $u_{ms}$ be the solutions of \cref{eq:weak_form} and \cref{eq:weak_form_ms}, respectively. Then the following error estimate holds:
\begin{equation}
\label{eq:energy_norm}
\begin{aligned}
    & \|(u-u_{ms})(\cdot, T)\|^2 + \int_0^T \|u-u_{ms}\|_a^2 dt \\
    & \le C \delta_{LSI}^2 \left( \|u_0\|_a^2 + \|u_{0,ms}\|_a^2 + \|f\|_{L^2(0,T;L^2(\Omega))}^2 \right) + \|(u-u_{ms})(\cdot, 0)\|^2,
    \end{aligned}
\end{equation}
where $C$ is a constant independent of the mesh size $H$.
\end{theorem}

\begin{proof}
Subtracting the semi-discrete equation \cref{eq:weak_form_ms} from the weak formulation \cref{eq:weak_form}, we satisfy the error equation:
\begin{equation}
    \langle (u-u_{ms})_t, v \rangle + a(u - u_{ms}, v) = 0, \quad \forall v \in V_{ms}^{(k)}.
\end{equation}
Using the decomposition $u - u_{ms} = \theta + \rho$ and the property of the elliptic projection $a(\theta, v) = 0$ for all $v \in V_{ms}^{(k)}$, we rewrite the equation for $\rho \in V_{ms}^{(k)}$ as:
\begin{equation}
    \langle \rho_t+\theta_t, v \rangle + a(\rho, v) = 0.
\end{equation}
Setting $v = \rho$, we obtain:
\begin{equation}
 \langle \theta_t+\rho_t, \rho \rangle + a(\rho, \rho) = 0.
\end{equation}
This can be written as 
\begin{equation}
     \langle \theta_t+\rho_t, \theta+\rho \rangle + a(\theta+\rho, \theta+\rho) = \langle \theta_t+\rho_t, \theta \rangle+a(\theta+\rho, \theta).
\end{equation}
This implies
\begin{equation}
\frac{1}{2} \frac{d}{dt} \|\theta+\rho\|^2 + \|\theta+\rho\|_a^2 = (\theta_t+\rho_t, \theta)+a(\theta+\rho, \theta).
\end{equation}
By the Cauchy--Schwarz inequality and Young's inequality, we have
\begin{equation}
\begin{aligned}
 \frac{1}{2} \frac{d}{dt} \|\theta+\rho\|^2 + \|\theta+\rho \|_a^2 
& \le \|\theta_t+\rho_t\| \|\theta \| + \|\theta_t+\rho_t\|_a \|\theta \|_a \\
&\le (\|\theta_t+\rho_t\| ) \|\theta \| + \frac{1}{2} \|\theta+\rho \|_a^2 + \frac{1}{2} \|\theta \|_a^2.
\end{aligned}
\end{equation}
So, we have
\begin{equation}
\frac{d}{dt}\|\theta+\rho\|^2 + \|\theta+\rho\|_a^2 \le 2(\|\theta_t+\rho_t\|)\|\theta\| + \|\theta \|_a^2.
\end{equation}
Integrating with respect to time, we have
\begin{equation*}
\begin{aligned}
\|(\theta+\rho)(\cdot, T)\|^2 & - \|(\theta+\rho)(\cdot, 0)\|^2 + \int_0^T \|\theta+\rho \|_a^2 \\
& \le \int_0^T 2(\|u_t\| + \|(u_{ms})_t\| )\| \theta \| + \int_0^T \|\theta \|_a^2.
\end{aligned}
\end{equation*}
By \cref{lemma_ut_estimation}, we can bound $u_t$ by the initial data and source function $f$, that is,
\begin{equation*}
\int_0^T \|u_t\|^2 \leq C \left( \|u_0\|_a^2 + \int_0^T \|f\|^2 \right).
\end{equation*}
Similarly, by \cref{eq:weak_form_ms}, we have
\begin{equation*}
 \int_0^T \|(u_{ms})_t\|^2 \leq C \left( \|u_{0,ms}\|_a^2 + \int_0^T \|f\|^2 \right).   
\end{equation*}
Thus, \cref{corollary_proj} and the above two inequalities give \cref{eq:energy_norm}.
This completes the proof.
\end{proof}

\begin{theorem}[$L^2$ Norm Estimate]
\label{thm:l2_norm}
Assume $f_t \in L^1(0,T; L^2(\Omega))$ and $u_{tt} \in L^1(0,T; L^2(\Omega))$. Let $u$ and $u_{ms}$ be the solutions of \eqref{eq:weak_form} and \eqref{eq:weak_form_ms}, respectively. Then:
\begin{equation}
\begin{aligned}
    \|(u-u_{ms}) & (\cdot, T)\|  \le \|u_0 - u_{0,ms}\|  \\
    & +C \delta_{LSI}^2 \left( \max_{0\le t\le T} \|(f-u_t)(t)\|+ \|f_t - u_{tt}\|_{L^1(0,T; L^2(\Omega))} \right).
    \end{aligned}
\end{equation}
\end{theorem}

\begin{proof}
We start with the error equation for $\rho$: 
\begin{equation*}
    \langle \rho_t, v \rangle + a(\rho, v) = -\langle \theta_t, v\rangle.
\end{equation*}
Choosing $v = \rho$, we have 
\begin{equation*}
    \|\rho\| \frac{d}{dt} \|\rho\| + \|\rho\|_a^2 = -\langle \theta_t, \rho \rangle. 
\end{equation*}

Since $\|\rho\|_a^2 \ge 0$, we have 
\begin{equation*}
    \|\rho\| \frac{d}{dt} \|\rho\| \le |\langle \theta_t, \rho \rangle | \le \|\theta_t\| \|\rho\|.
\end{equation*}
Canceling $\|\rho\|$ (for $\rho \ne 0$), we get the differential inequality:
\begin{equation}
    \frac{d}{dt}\|\rho\| \le \|\theta_t\|.
\end{equation}
Integrating over $[0, T]$ gives:
\begin{equation}
    \|\rho(\cdot, T)\| \le \|\rho(\cdot, 0)\| + \int_0^T \|\theta_t\| dt.
\end{equation}
To estimate $\|\theta_t\|$, we take the time derivative of the definition of the elliptic projection. Let $\theta_t = (u - \mathcal{P}u)_t = u_t - \mathcal{P}u_t$. Note that $\mathcal{P}$ is time-independent because the coefficient $\kappa(x)$ is time-independent. Thus, $\mathcal{P}u_t$ is simply the elliptic projection of $u_t$.
Applying \cref{lemma_proj_para} (the $L^2$-estimate part) to the function $u_t$, we get:
\begin{equation}
    \|\theta_t\| = \|u_t - \mathcal{P}u_t\| \le C \delta_{LSI}^2 \|\mathcal{L}u_t\| = C \delta_{LSI}^2 \|(f - u_t)_t\|.
\end{equation}
The above inequality implies that
\begin{equation}
\begin{aligned}
    \|\rho(\cdot, T)\| & \le \|\rho(\cdot, 0)\| + C \delta_{LSI}^2 \int_0^T \|(f - u_t)_t\| dt \\
    & \le \| u_0 - u_{0,ms} \| + \| \theta_0  \|+ C \delta_{LSI}^2 \int_0^T \|(f - u_t)_t\| dt.
\end{aligned}
\end{equation}
Finally, we have
\begin{equation}
\begin{aligned}
 \|(u - & u_{ms}) (\cdot, T)\|  \leq \|(u - \mathcal{P} u)(\cdot, T)\| + \|(\mathcal{P} u - u_{ms})(\cdot, T)\| \\
& \leq C \delta_{LSI}^2 \left( \max_{0 \leq t \leq T} \|(f - u_t)\| + \int_0^T \| (f_t - u_{tt})\| dt \right)
 + \|u_0 - u_{0,ms}\|.
\end{aligned}
\end{equation}
This completes the proof.
\end{proof}

\section{Contrast-independent partially explicit time discretization} 
\label{sec:temporalsplitting}
To efficiently solve the semidiscrete parabolic problem \cref{eq:weak_form_ms}, we employ a temporal splitting algorithm that allows for a time step independent of the high-contrast variations in the permeability field \cite{chung2021contrast}.

\subsection{Decomposition of the local multiscale space}
We decompose the global multiscale space $V_{\mathrm{ms}}^{(k)}$ into two orthogonal subspaces: an implicit space $V_{H,1}$ and an explicit space $V_{H,2}$.

\textbf{Explicit subspace ($V_{H,2}$):} This space is constructed to capture the dominant multiscale modes of the solution. Based on the local multiscale space $V_{\mathrm{ms},i}^k$, we select the basis functions corresponding to the largest eigenvalues of the local inverse operator $\mathcal{L}_i^{-1}$. Because these low-frequency modes have bounded energy norms relative to their $L^2$ norms, they can be treated explicitly.

\textbf{Implicit subspace ($V_{H,1}$):} This space acts as a correction space and consists of the remaining basis functions generated by the LSI method, corresponding to the smaller eigenvalues of the local inverse operator $\mathcal{L}_i^{-1}$. These high-frequency modes govern the microscopic behavior and require implicit temporal treatment to maintain stability without restrictive time-step constraints.

We define the local multiscale explicit and implicit spaces as $V_{H,2} = \bigoplus_{i=1}^{N_c} V_{H,2,i}$ and $V_{H,1} = \bigoplus_{i=1}^{N_c} V_{H,1,i}$, respectively. Here, the local subspaces $V_{H,2,i} = \text{span}\{\hat{\phi}_i^j\}_{j=1}^{l_i}$ and $V_{H,1,i} = \text{span}\{\hat{\phi}_i^j\}_{j=l_i+1}^{L_i}$ are both generated using the LSI method, where $l_i$ is the number of basis in local explicit subspace. Globally, the space is defined as $V_{ms}^{(k)} = V_{H,1} \oplus V_{H,2}$, allowing us to represent the numerical solution at time step $n$ as $u_{ms}^n = u_{H,1}^n + u_{H,2}^n$.

\begin{remark}
The appropriate selection of $l_{i}$, the dimension of the local explicit subspace $V_{H,2,i}$, is crucial for balancing the computational efficiency and stability of the temporal splitting scheme. In highly heterogeneous media, particularly those with high permeability contrasts such as fractures or channelized networks, the spectrum of the local inverse operator $\mathcal{L}_{i}^{-1}$ typically exhibits a sharp decline, creating a distinct ``spectral gap.'' This gap naturally separates the dominant macroscopic modes (corresponding to the highly conductive features) from the microscopic background modes. To optimize the splitting scheme, $l_{i}$ should be chosen precisely at this spectral gap. By doing so, we ensure that all critical, slow-decaying multiscale components are captured by the explicit subspace $V_{H,2}$ and advanced efficiently. Meanwhile, the remaining high-frequency modes are relegated to the implicit subspace $V_{H,1}$, thereby guaranteeing contrast-independent stability without unnecessary computational overhead.
\end{remark}

\subsection{The temporal splitting scheme}
We consider a fixed time step $\tau$ and define $t^n = n\tau$. To advance the solution from $t^n$ to $t^{n+1}$, we adopt a partially explicit scheme where the equations are decoupled such that $V_{H,1}$ is treated implicitly and $V_{H,2}$ is treated explicitly. 

Find $\{u_{H,1}^n\}_{n=1}^N \in V_{H,1}$ and $\{u_{H,2}^n\}_{n=1}^N \in V_{H,2}$ such that for a parameter $\omega \in [0, 1]$:
\begin{equation} \label{eq:temporal:implicit}
\langle u_{H,1}^{n+1},v \rangle = \langle u_{H,1}^n,v \rangle - \langle u_{H,2}^n - u_{H,2}^{n-1},v \rangle - \tau a(u_{H,1}^{n+1} + u_{H,2}^n, v), \quad \forall v \in V_{H,1}
\end{equation}
\vspace{-0.5cm}
\begin{equation} \label{eq:temporal:explicit}
\langle u_{H,2}^{n+1},v \rangle = \langle u_{H,2}^n,v \rangle - \langle u_{H,1}^n - u_{H,1}^{n-1},v \rangle - \tau a((1-\omega)u_{H,1}^n + \omega u_{H,1}^{n+1} + u_{H,2}^n, v), \quad \forall v \in V_{H,2}
\end{equation}
This formulation ensures that the computationally expensive implicit inversion is restricted to the much smaller subspace $V_{H,1}$, significantly reducing the computational burden.

\subsection{Stability condition}
\label{subcec:stacond}
The primary advantage of this splitting method is that, given a suitable choice of $V_{H,2}$, the stability condition relies solely on the properties of $V_{H,2}$ and is rendered independent of the high-contrast media.

Let $\gamma$ be the strengthened Cauchy-Schwarz constant between the two spaces, defined as:
\begin{equation} \nonumber
\gamma := \sup_{v_1 \in V_{H,1}, v_2 \in V_{H,2}} \frac{\langle v_1, v_2 \rangle}{\|v_1\| \|v_2\|} < 1.
\end{equation}

The partially explicit scheme is stable provided the time step $\tau$ satisfies the condition:
\begin{equation}
\tau \sup_{v \in V_{H,2}} \frac{\|v\|_a^2}{\|v\|^2} \le \frac{1-\gamma^2}{2-\omega}.
\end{equation}
By specifically constructing $V_{H,2}$ from the eigenfunctions associated with the largest eigenvalues of $\mathcal{L}_i^{-1}$, the supremum $\sup_{v \in V_{H,2}} \frac{\|v\|_a^2}{\|v\|^2}$ is bounded independently of the contrast $\kappa_{\max}/{\kappa_{\min}}$.

\subsection{Upper bound for the explicit subspace Rayleigh quotient}

To guarantee that the stability condition in \cref{subcec:stacond} is independent of the high-contrast features of $\kappa(x)$, we must establish an upper bound for the Rayleigh quotient $\sup_{v \in V_{H,2}} \frac{\|v\|_a^2}{\|v\|^2}$. Our analysis relies on the following classical lemma. For a detailed proof, we refer the reader to \cite{chung2021contrast}.
\begin{lemma}
    Assume that there is a constant $0 < \beta_{j,m}<1$ such that
    \begin{equation*}
       \beta_{j,m} := \sup_{v_j \in V_{H,2,j}, v_m \in V_{H,2,m}} \frac{\langle v_j, v_m \rangle }{\|v_j\| \|v_m\|}, \quad \forall j,m=1, \cdots, N_c, \; j \neq m. 
    \end{equation*}
    Let $\beta = \max_{j,m} \beta_{j,m}$, then we have
    \begin{equation}
        \sup_{v \in V_{H,2}} \frac{\|v\|_a^2}{\|v\|^2} \leq N_c (1-\beta)^{-p} \sup_{1 \leq j \leq N_c} \sup_{v_j \in V_{H,2,j}} \frac{\|v_j\|_a^2}{\|v_j\|^2}
    \end{equation}
    where $p$ is the smallest integer greater than or equal to $\log_2 N_c$.
\end{lemma}

This lemma demonstrates how the maximum Rayleigh quotient in the global multiscale space $V_{H,2}$ can be controlled by those in the local multiscale spaces $V_{H,2,i}$. According to the subspace iteration method, the maximum Rayleigh quotient in $V_{H,2,i}$ is approximately $\lambda_i^{l_i}$, the $l_i$-th eigenvalue of the inverse operator $\mathcal{L}_i^{-1}$. 
\begin{lemma}
Define $\mathcal{Q}_i = \sup_{v_i \in V_{H,2,i}} \frac{\|v_i\|_a^2}{\|v_i\|^2}$. Applying the localized standard subspace iteration method yields the bound
\begin{equation}
    | \mathcal{Q}_i - \frac{1}{\lambda_i^{l_i} } | = \mathcal{O} \left( \left| \frac{\lambda_i^{L_i+1}}{\lambda_i^{l_i}} \right|^{2k} \right).
\end{equation}
In contrast, for the localized Krylov subspace iteration method, we obtain the error bound
\begin{equation}
    | \mathcal{Q}_i - \frac{1}{\lambda_i^{l_i} } | = \lambda_i^1 \mathcal{O} \left( \left| \frac{1}{C_{k-l_i}\left( 1+2 \gamma_i^{l_i}\right)} \right|^{2} \right),
\end{equation}
where $\gamma_i^{l_i} = \frac{\lambda_i^{l_i} - \lambda_i^{l_i+1}}{ \lambda_i^{l_i+1}}$ and $C_{k-l_i}(x)$ denotes the Chebyshev polynomial of the first kind of degree $k-l_i$. For $x > 1$, the Chebyshev polynomial $C_{k-l_i}(x)$ grows exponentially.
\end{lemma}

The proof of this lemma can be found in \cite{saad2011numerical}. This lemma establishes how the maximum Rayleigh quotient within the subspace $V_{H,2,i}$ approximates $\frac{1}{\lambda_i^{l_i}}$. Interestingly, our numerical observations suggest that the largest eigenvalues of the inverse operator $\mathcal{L}_i^{-1}$ are robust with respect to the contrast, despite the current lack of a rigorous theoretical proof. Consequently, we conclude that, provided the iteration count $k$ and the subspace dimension $l_i$ are chosen appropriately, the stability of the temporal splitting scheme \cref{eq:temporal:implicit}--\cref{eq:temporal:explicit} remains independent of the contrast.

\section{Numerical results}
\label{sec:numresult}
In this section, we present several numerical examples to evaluate the performance of the proposed LSI method for solving the multiscale parabolic equation \cref{eq:parabolic_strong}. To establish the localized subdomains for our method, we systematically enlarge each coarse element $K_i \in \mathcal{T}_H$ by appending adjacent coarse grid blocks. Specifically, the oversampled regions $K_i^m$ are defined recursively as follows:
$$K_i^0 := K_i,$$
$$K_i^m := \text{int}\left( \bigcup_{T \in \mathcal{T}_H, T \cap \overline{K}_i^{m-1} \neq \emptyset} T \right), \quad m = 1, 2, 3, \dots$$
Here, the integer $m$ denotes the number of oversampling layers. In our simulations, the local target subdomain is defined as $\omega_i = K_i^m$. 

To accurately capture the microscale features and evaluate the local problems, a sufficiently fine mesh $\mathcal{T}_h$ is employed. This fine-scale grid also serves to compute the highly resolved reference solutions. For all subsequent experiments, the reported approximation errors are computed relative to these fine-grid reference solutions. Furthermore, the numerical tests are executed on a standard desktop computer equipped with an Intel Core i7 processor (3.4 GHz) and 16 GB of RAM.

\subsection{Accuracy and convergence of the LSI method}
\label{sec:ex1}

In our first numerical experiment, we investigate the spatial approximation accuracy and the convergence behavior of the proposed LSI method. We define the computational domain as a unit square, i.e., $\Omega = (0, 1)^2$. The spatial domain is partitioned using a uniform coarse grid with a mesh size of $H = 1/10$. To fully resolve the underlying microscale heterogeneities, each coarse element is partitioned, yielding a fine-scale mesh size of $h = 1/100$.

The spatial heterogeneity of the porous medium is characterized by a multiscale permeability field $\kappa(x)$, which is visually depicted in \cref{fig:kappa_ex1}. This specific configuration features complex, highly oscillatory microstructures that pose significant challenges for standard finite element methods. For the time-dependent components of the parabolic equation \cref{eq:parabolic_strong}, we set the final simulation time to $T = 0.1$ and employ a uniform time step size $\tau$ (e.g., $\tau = 10^{-4}$) for the temporal integration. The source term is defined as $f(x,t) = 2\pi^2 \sin{\pi x_1} \sin{\pi x_2}$, and the initial condition is set to $u_0(x) = 0$.

\begin{figure}[htbp]
    \centering
    \includegraphics[width=0.5\textwidth]{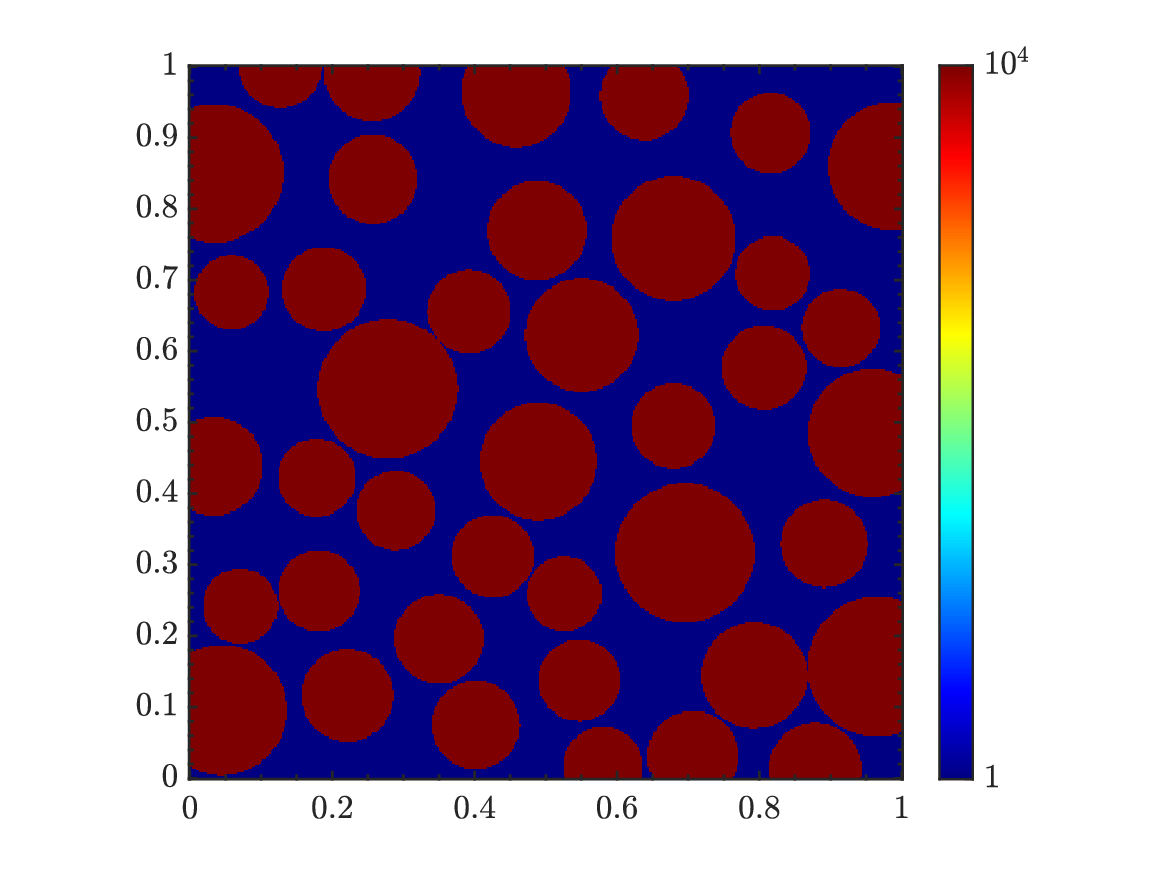}
    \caption{The spatial distribution of the multiscale permeability field $\kappa(x)$.}
    \label{fig:kappa_ex1}
\end{figure}

For the LSSI method, the initial local space is spanned by four standard bilinear shape functions defined on the central coarse element $K_i$. In contrast, for the LKSI method, the initial local function is simply chosen as a constant function restricted to the central coarse element $K_i$. In both approaches, we set the number of oversampling layers to $m=4$ and perform $n=4$ iteration steps. \Cref{fig:solution_comparison_ex1} presents the profiles of the fine-grid reference solution alongside the numerical solutions generated by the LSSI and LKSI methods at the final simulation time. 
Visual comparisons indicate that both proposed methods accurately capture the complex multiscale features of the global solution. Despite employing a highly localized, low-dimensional approximation space, the LSSI and LKSI solutions exhibit excellent visual agreement with the high-fidelity reference solution.

\begin{figure}[htbp]
    \centering
    \begin{subfigure}[b]{0.32\textwidth}
        \centering
        \includegraphics[width=\textwidth]{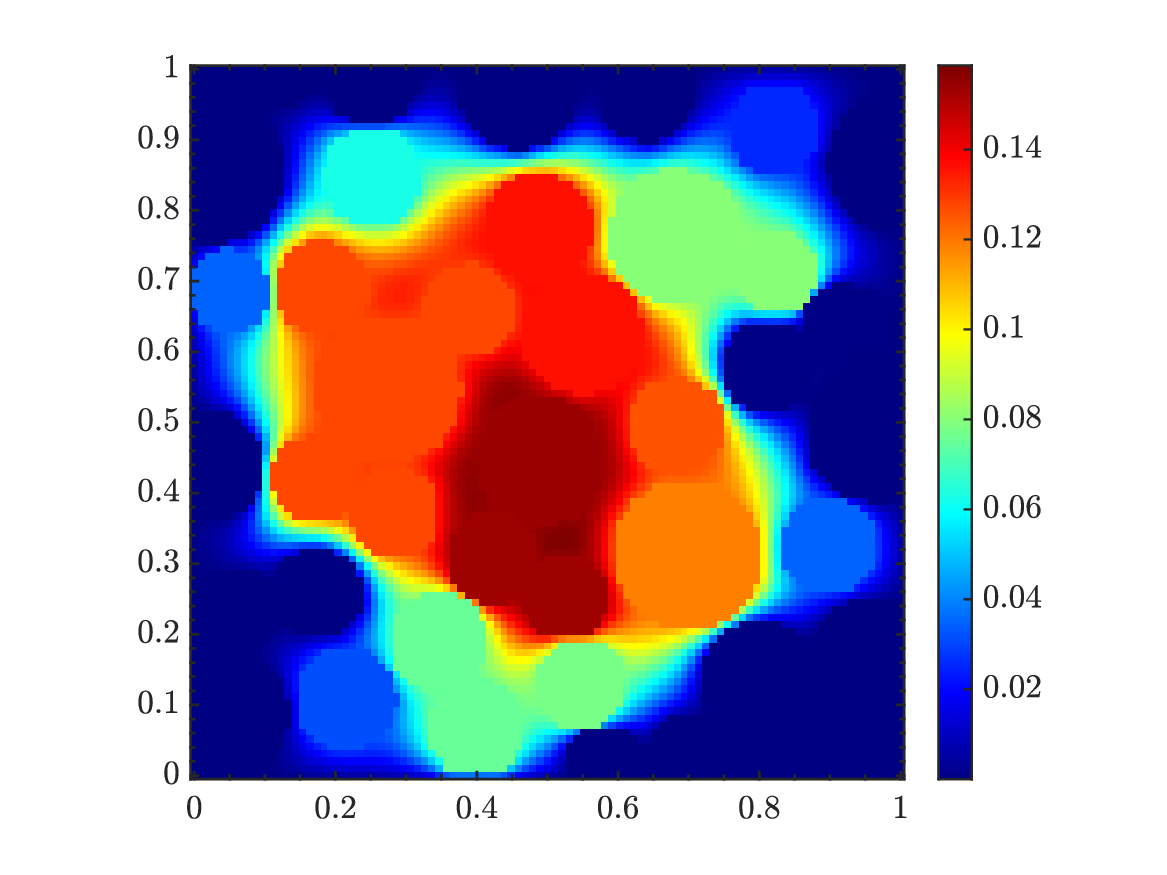}
        \caption{Reference solution}
    \end{subfigure}
    \hfill
    \begin{subfigure}[b]{0.32\textwidth}
        \centering
        \includegraphics[width=\textwidth]{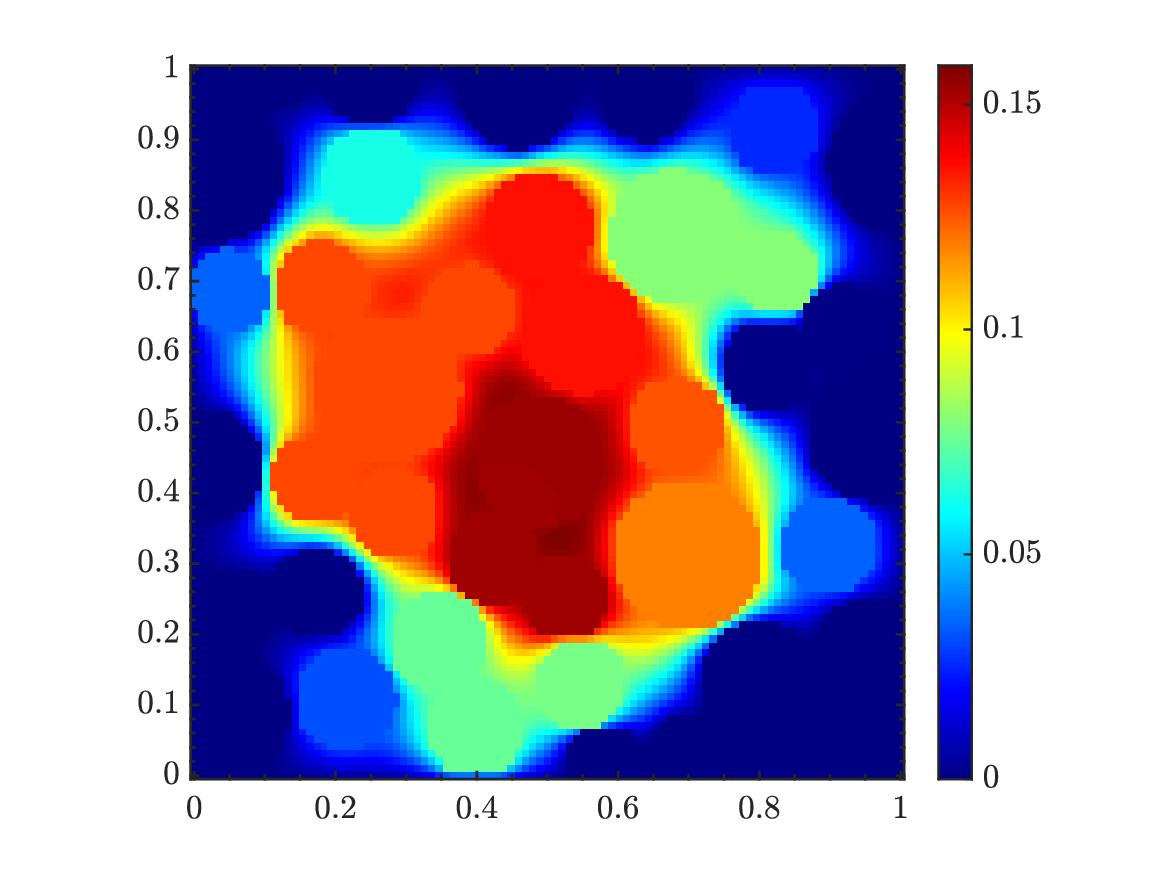}
        \caption{LSSI solution}
    \end{subfigure}
    \hfill
    \begin{subfigure}[b]{0.32\textwidth}
        \centering
        \includegraphics[width=\textwidth]{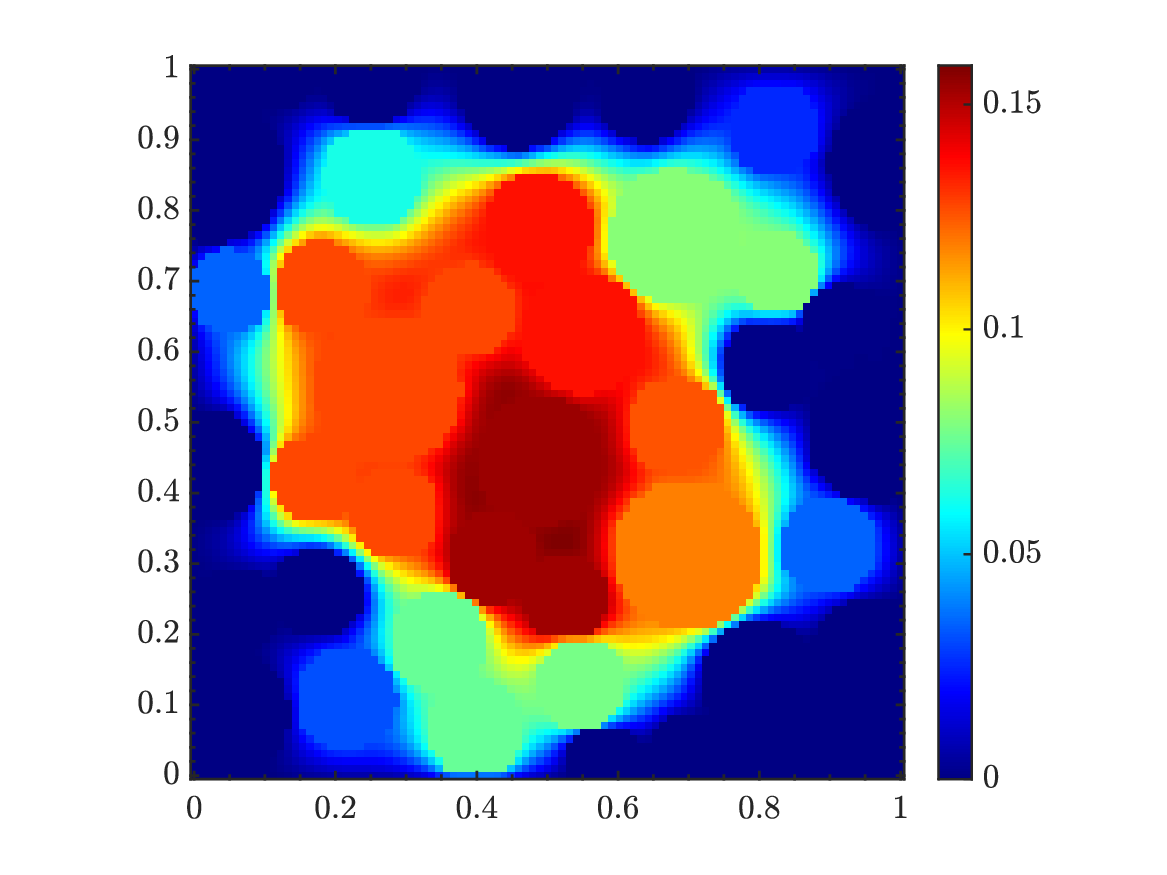}
        \caption{LKSI solution}
    \end{subfigure}
    \caption{Comparison of the numerical solutions at the final time $T$. The multiscale solutions are computed with $m=4$ oversampling layers and $n=4$ iteration steps.}
    \label{fig:solution_comparison_ex1}
\end{figure}


\begin{table}[htbp]
    \centering
    \caption{Comparison of LSSI and LKSI methods with different iteration steps.}
    \label{tab:ex1_results}
    \begin{tabular}{lccccc}
        \toprule
        Method & Energy error & $L^2$ error & CPU time (s) & DoF & NoLP \\
        \midrule
        LSSI-1 & $1.76 \times 10^{-2}$ & $1.04 \times 10^{-3}$ & $3.29$ & $400$ & $400$ \\
        LSSI-2 & $8.32 \times 10^{-3}$ & $4.30 \times 10^{-4}$ & $4.26$ & $400$ & $800$ \\
        LSSI-3 & $7.23 \times 10^{-3}$ & $3.39 \times 10^{-4}$ & $4.93$ & $400$ & $1200$ \\
        LSSI-4 & $7.77 \times 10^{-3}$ & $4.19 \times 10^{-4}$ & $5.99$ & $400$ & $1600$ \\
        \midrule
        LKSI-1 & $1.52 \times 10^{-2}$ & $8.59 \times 10^{-4}$ & $4.42$ & $100$ & $100$ \\
        LKSI-2 & $7.62 \times 10^{-3}$ & $4.13 \times 10^{-4}$ & $4.54$ & $200$ & $200$ \\
        LKSI-3 & $6.97 \times 10^{-3}$ & $3.88 \times 10^{-4}$ & $4.76$ & $300$ & $300$ \\
        LKSI-4 & $6.75 \times 10^{-3}$ & $4.00 \times 10^{-4}$ & $5.11$ & $400$ & $400$ \\
        LKSI-5 & $6.21 \times 10^{-3}$ & $3.14 \times 10^{-4}$ & $6.46$ & $500$ & $500$ \\
        LKSI-6 & $5.74 \times 10^{-3}$ & $2.41 \times 10^{-4}$ & $7.41$ & $600$ & $600$ \\
        \bottomrule
    \end{tabular}
\end{table}

\Cref{tab:ex1_results} presents the numerical results, comparing the performance of the LSSI and LKSI methods across different iteration steps in terms of energy error, $L^2$ error, computational time for basis construction (CPU time), degrees of freedom (DoF), and the number of local problems (NoLP). Here, LSSI-$n$ and LKSI-$n$ denote the corresponding methods with $n$ iteration steps. For all test cases, the number of oversampling layers is consistently set to $m=4$.
As observed from the table, both the LSSI and LKSI methods exhibit similar convergence behaviors: the energy error and $L^2$ error decrease significantly as the iteration step $n$ increases from 1 to 3, and then stabilize when $n=4$, indicating that both methods approach their optimal accuracy limits under the current spatial grid resolution. 

However, a detailed comparison highlights the superior efficiency of the LKSI method. For any given iteration step $n$, LKSI requires solving substantially fewer local problems than LSSI. For instance, at the fourth iteration step, LSSI-4 requires solving 1600 local problems, whereas LKSI-4 only requires 400. Consequently, while maintaining the same global DoF (400), LKSI-4 not only achieves a slightly better accuracy in both energy and $L^2$ norms but also requires less CPU time ($5.11$s) compared to LSSI-4 ($5.99$s). 
This underscores the superior computational efficiency of the LKSI method, which delivers comparable or higher accuracy while drastically reducing the burden of solving local eigenvalue problems.

\begin{figure}[htbp]
    \centering
    \begin{subfigure}{0.48\textwidth}
        \centering
        \includegraphics[width=\textwidth]{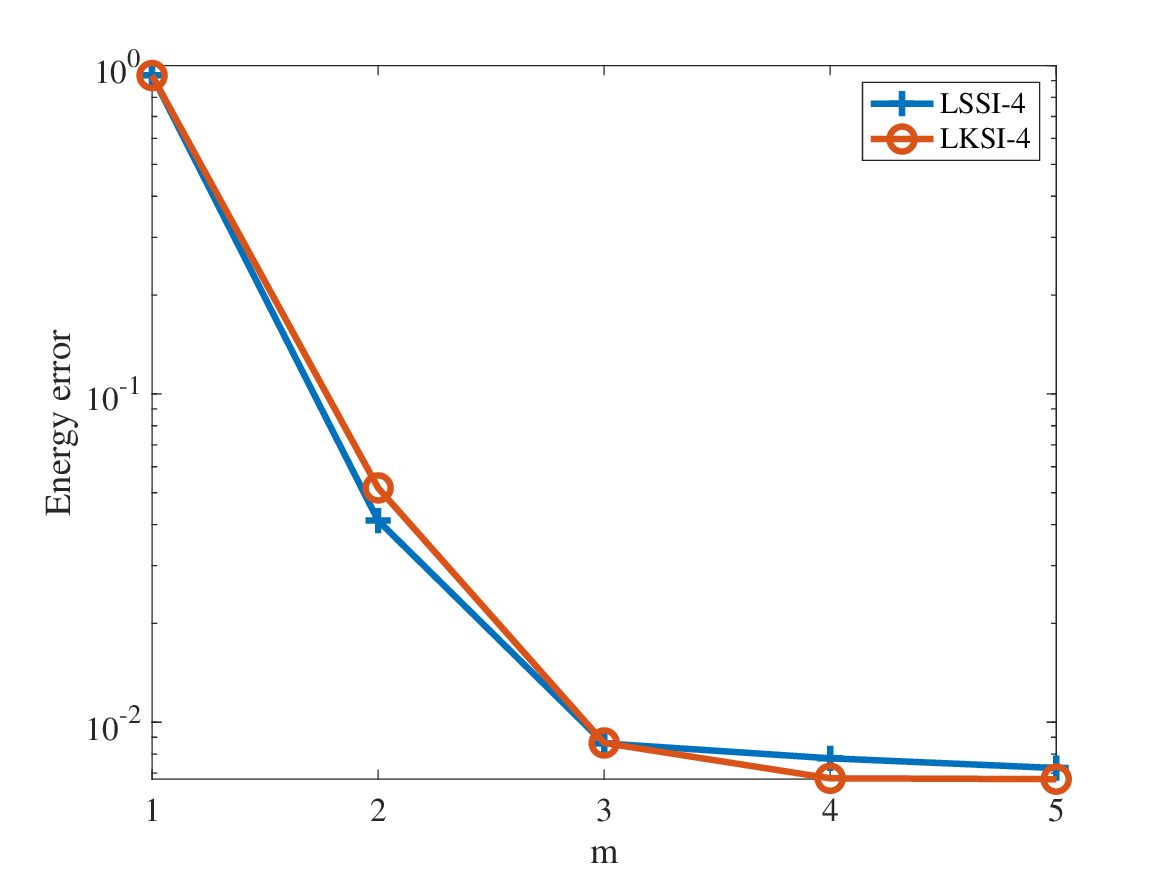}
        \label{fig:energy_ex1}
    \end{subfigure}
    \hfill 
    \begin{subfigure}{0.48\textwidth}
        \centering
        \includegraphics[width=\textwidth]{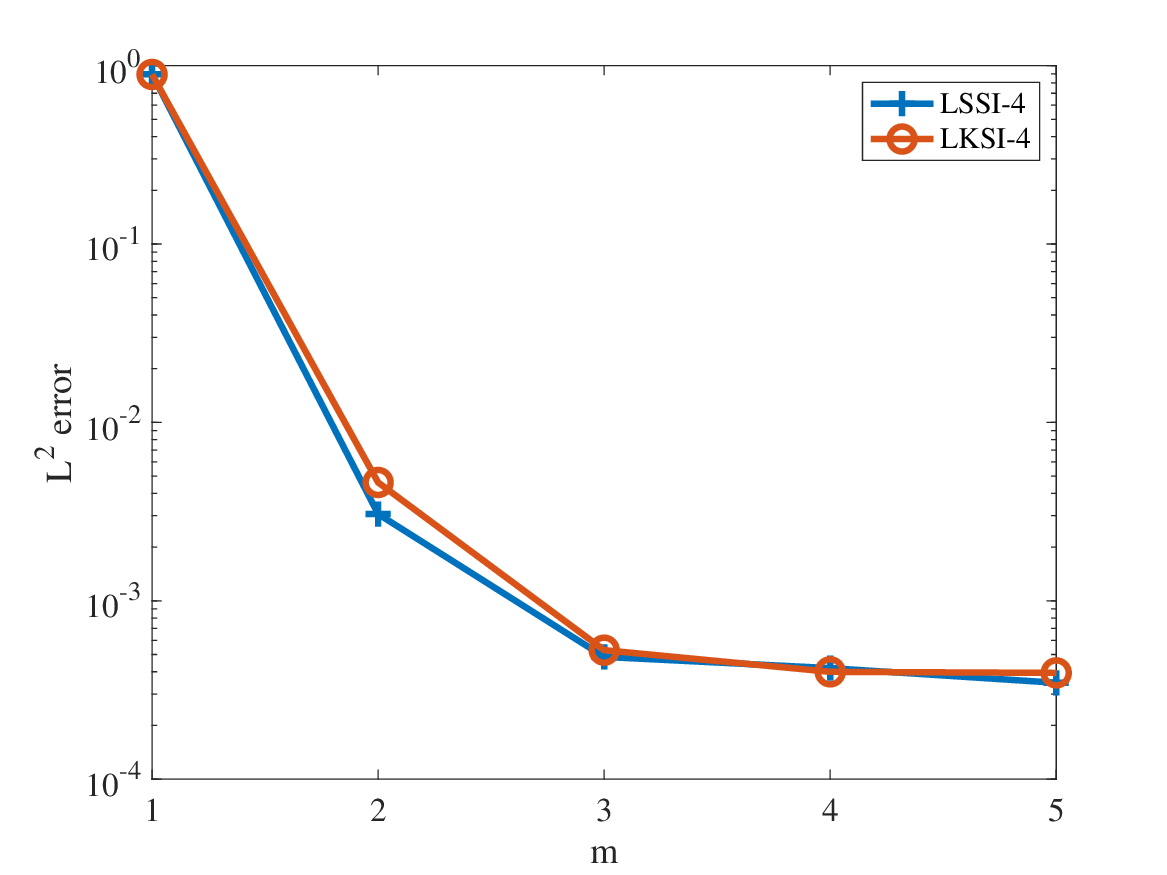}
        \label{fig:l2_ex1}
    \end{subfigure}
    
    \caption{Comparison of the Energy error and $L^2$ error against the number of oversampling layers $m$.}
    \label{fig:error_vs_m_ex1}
\end{figure}

\cref{fig:error_vs_m_ex1} further demonstrates the Energy and $L^2$ errors for the LSSI and LKSI methods as the oversampling parameter $m$ increases from 1 to 5. Both methods display a rapid exponential decrease in error as the localized domain expands.

\subsection{Performance and stability of the temporal splitting scheme} In this section, we investigate the performance of the partially explicit temporal splitting scheme introduced in \cref{sec:temporalsplitting}. Specifically, we aim to verify that the stability condition of our proposed method remains independent of the high-contrast features of the permeability field $\kappa(x)$. To ensure a consistent evaluation, we retain the baseline parameter settings established in section \ref{sec:ex1}. Furthermore, we employ $m=4$ oversampling layers and $n=4$ iteration steps for the local multiscale space construction. The primary objective here is to evaluate the approximation accuracy and numerical stability over the simulation time $T=0.1$. By comparing the fully implicit method with the temporal splitting scheme, we aim to demonstrate that the partially explicit approach maintains high accuracy in both energy and $L^2$ norms while efficiently resolving the high-frequency microscopic behaviors without the need for restrictive time-step constraints.

\cref{fig:error_time} depicts the evolution of the Energy and $L^2$ errors over the simulation time interval $(0, 0.1]$ for both the LSSI and LKSI methods. We compare the standard fully implicit time discretization with the proposed partially explicit temporal splitting scheme. For these simulations, the local explicit subspace $V_{H,2} $ are constructed using $l_{i}=2$ basis functions. As illustrated in \cref{fig:error_time} , the energy error trajectories produced by the temporal splitting scheme are virtually indistinguishable from those of the fully implicit scheme throughout the entire temporal domain. This excellent agreement verifies that treating the low-frequency modes explicitly does not introduce any significant additional temporal truncation error. Consequently, the partially explicit scheme successfully preserves the optimal spatial approximation accuracy of the LSI method while maintaining a contrast-independent stability condition, significantly reducing the computational overhead associated with implicit inversions.

\begin{figure}[htbp]
    \centering
    \begin{minipage}{0.48\textwidth}
        \centering
        \includegraphics[width=\linewidth]{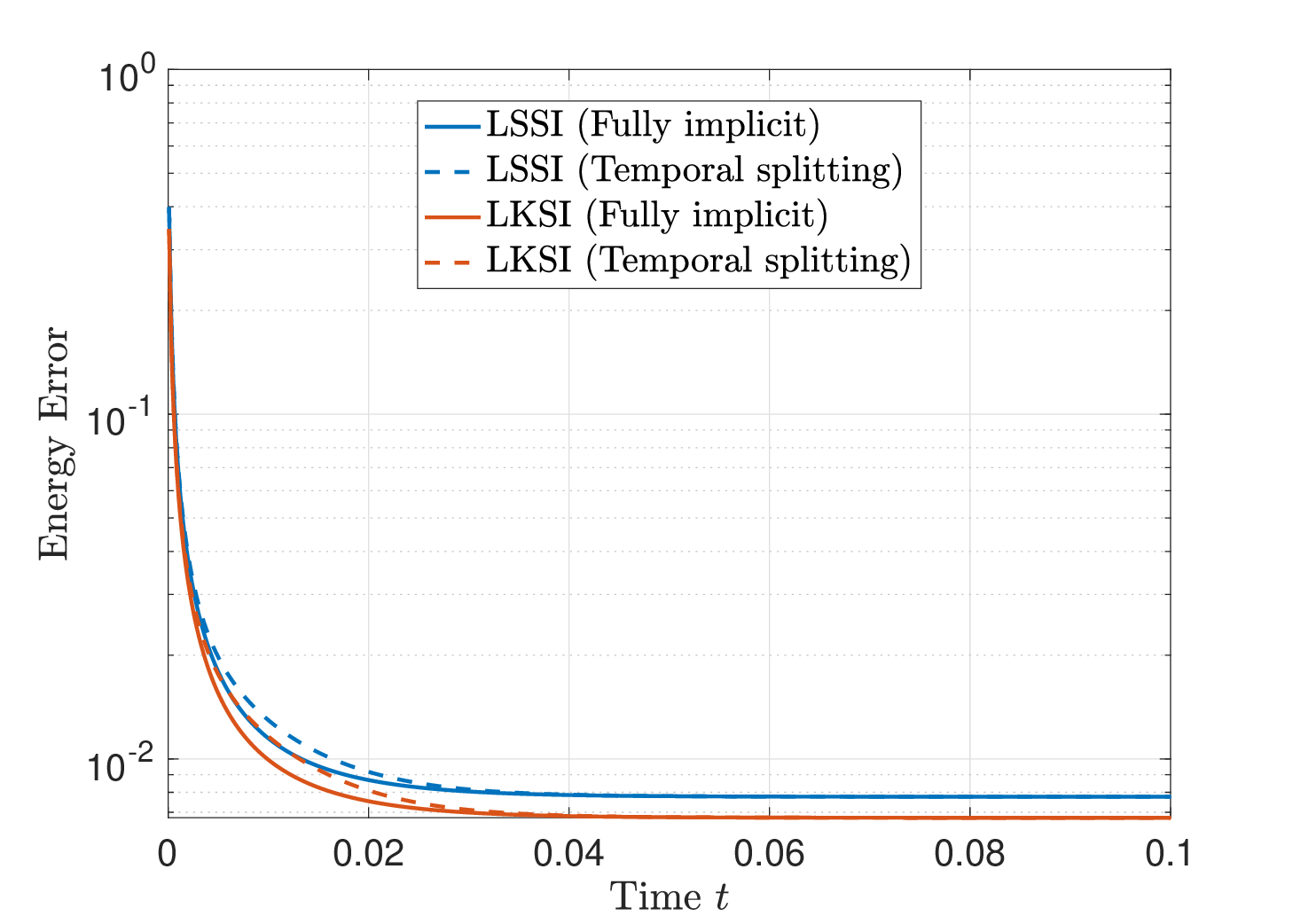}
    \end{minipage}\hfill
    \begin{minipage}{0.48\textwidth}
        \centering
        \includegraphics[width=\linewidth]{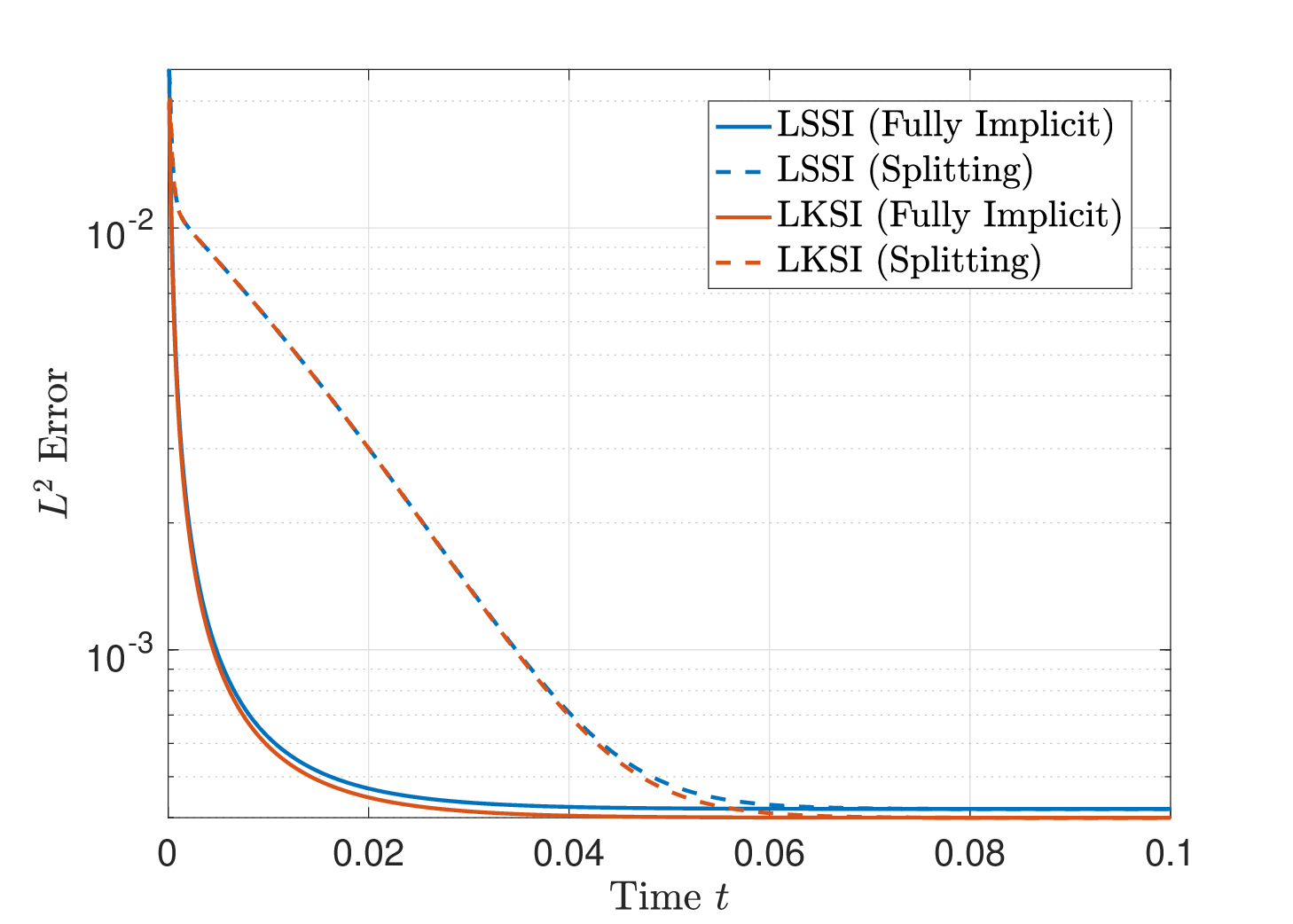}
    \end{minipage}
    \caption{Comparison of the Energy error and $L^2$ error evolution over time $t \in (0, 0.1]$ between the fully implicit scheme and the temporal splitting scheme.}
    \label{fig:error_time}
\end{figure}

To further validate the robustness of the temporal splitting scheme, we investigate the approximation errors and the maximum Rayleigh quotient of the explicit subspace under varying contrast levels. As established in section \ref{subcec:stacond}, the stability condition relies solely on the properties of the explicit subspace $V_{H,2}$ and is rendered independent of the high-contrast media. \cref{tab:contrast_variation} presents the Energy error, $L^2$ error, and the explicit subspace Rayleigh quotient (RQ) for the LSSI and LKSI methods as the permeability contrast ($\kappa_{max}/\kappa_{min}$) increases from $10^4$ to $10^8$.

As the contrast increases, both the Energy and $L^2$ errors remain remarkably stable, even exhibiting a slight decrease across all tested configurations (LSSI-1, LSSI-2, LSSI-4, and LKSI-4). This demonstrates the robustness of the constructed multiscale basis functions in effectively capturing high-contrast features. Most crucially, the maximum Rayleigh quotient in the global multiscale space $V_{H,2}$ remains bounded and does not grow proportionally with the increasing contrast. Across all methods, the RQ values plateau between $1.09 \times 10^4$ and $1.57 \times 10^4$, even at a severe contrast ratio of $10^8$. This numerical evidence strongly supports the theoretical analysis provided in Section 5.4, confirming that the stability of the partially explicit temporal splitting scheme is unconditionally independent of the medium's contrast.

\begin{table}[htbp]
\centering
\caption{Comparison of Energy error, $L^2$ error, and explicit subspace Rayleigh quotient (RQ) under varying contrast levels.}
\label{tab:contrast_variation}
\resizebox{\textwidth}{!}{
\begin{tabular}{c|cccc|cccc|cccc}
\toprule
\multirow{2}{*}{\textbf{Contrast}} & \multicolumn{4}{c|}{\textbf{Energy Error}} & \multicolumn{4}{c|}{\textbf{$L^2$ Error}} & \multicolumn{4}{c}{\textbf{Rayleigh Quotient (RQ)}} \\
\cmidrule(lr){2-5} \cmidrule(lr){6-9} \cmidrule(lr){10-13}
 & LSSI-1 & LSSI-2 & LSSI-4 & LKSI-4 & LSSI-1 & LSSI-2 & LSSI-4 & LKSI-4 & LSSI-1 & LSSI-2 & LSSI-4 & LKSI-4 \\
\midrule
$10^4$ & 1.76e-2 & 8.32e-3 & 7.77e-3 & 6.75e-3 & 1.04e-3 & 4.30e-4 & 4.19e-4 & 4.00e-4 & 13331 & 11520 & 11422 & 15744 \\
$10^5$ & 1.58e-2 & 7.48e-3 & 7.59e-3 & 6.62e-3 & 9.74e-4 & 3.69e-4 & 4.07e-4 & 3.90e-4 & 13299 & 11482 & 11288 & 15203 \\
$10^6$ & 1.52e-2 & 7.07e-3 & 7.57e-3 & 6.54e-3 & 9.65e-4 & 3.38e-4 & 4.05e-4 & 3.82e-4 & 13291 & 11470 & 11160 & 14825 \\
$10^7$ & 1.51e-2 & 6.91e-3 & 7.48e-3 & 6.47e-3 & 9.63e-4 & 3.30e-4 & 3.93e-4 & 3.74e-4 & 13287 & 11463 & 11061 & 14439 \\
$10^8$ & 1.51e-2 & 6.89e-3 & 7.41e-3 & 6.42e-3 & 9.62e-4 & 3.31e-4 & 3.85e-4 & 3.67e-4 & 13286 & 11458 & 10930 & 14292 \\
\bottomrule
\end{tabular}
}
\end{table}

\subsection{Application to realistic fractured media}
In this final numerical experiment, we evaluate the performance of the proposed LSI methods on realistic geological fractured media. Fractured porous media are characterized by highly complex, channelized networks with high permeability contrasts.
To construct the heterogeneous permeability fields, we utilize the GeoCrack dataset, a high-resolution collection of fracture edges in geological outcrops \cite{ansari2024geocrack}. We extract three distinct fracture network samples to serve as our permeability fields, denoted as $\kappa_1$, $\kappa_2$, and $\kappa_3$ (depicted in \cref{fig:ex3_kappa}). 
The fractures are assigned permeability values orders of magnitude higher than the surrounding rock matrix, effectively simulating highly conductive flow channels.

\begin{figure}[htbp]
    \centering
    \begin{subfigure}{0.32\textwidth}
        \includegraphics[width=\textwidth]{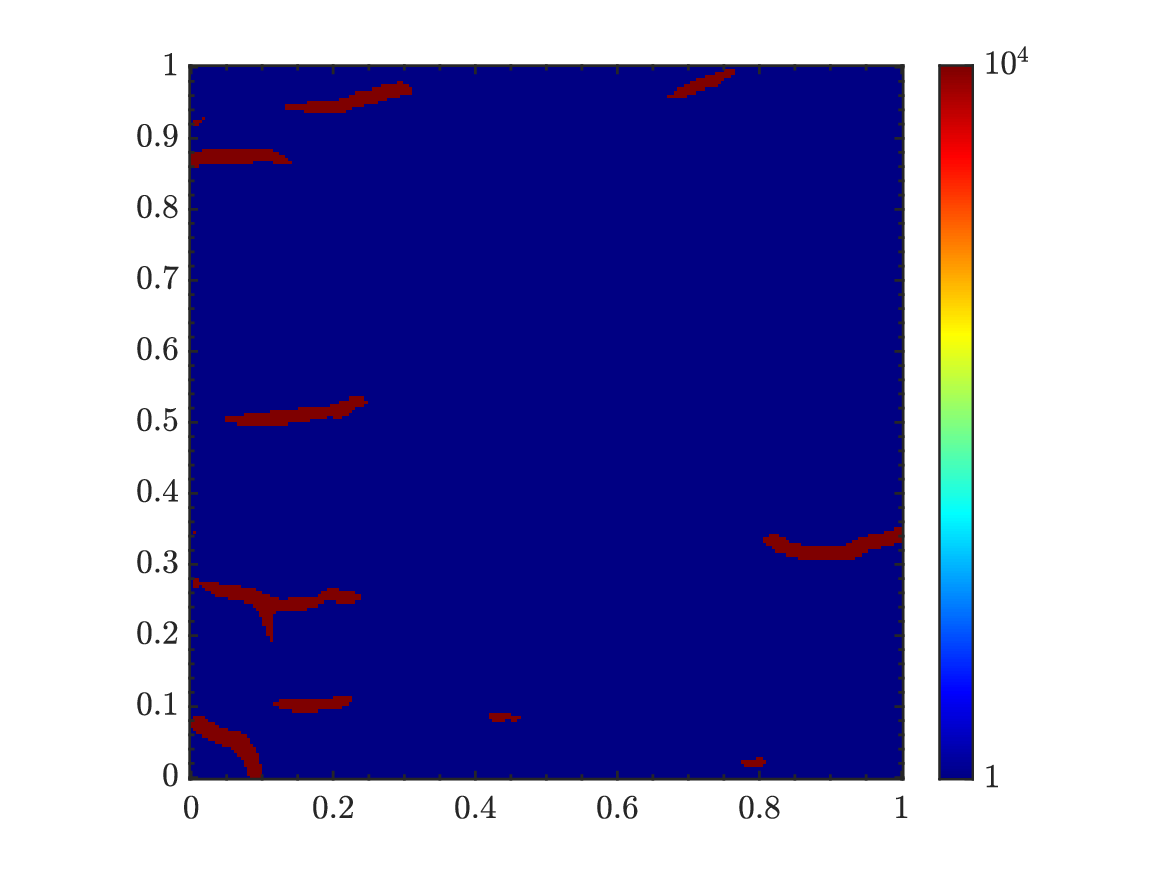}
        \caption{$\kappa_1$}
    \end{subfigure}
    \hfill
    \begin{subfigure}{0.32\textwidth}
        \includegraphics[width=\textwidth]{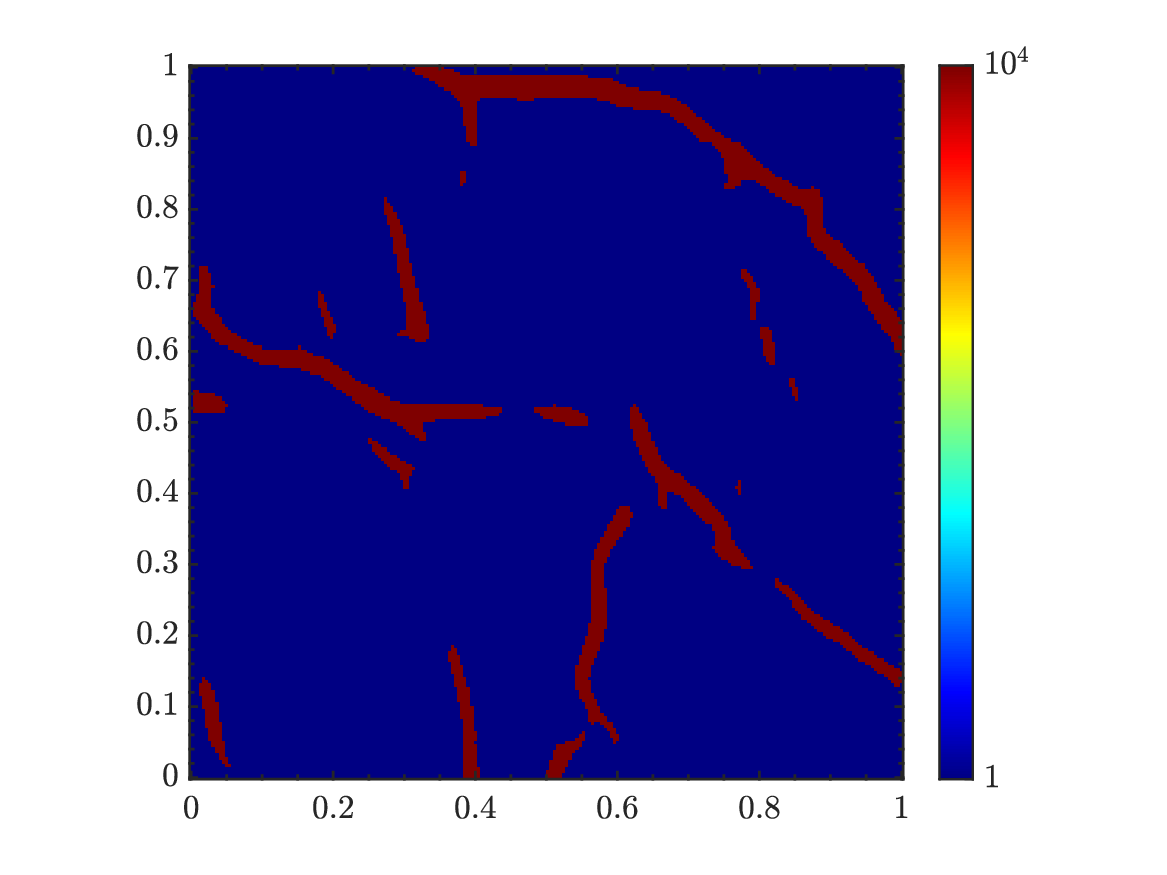}
        \caption{$\kappa_2$}
    \end{subfigure}
    \hfill
    \begin{subfigure}{0.32\textwidth}
        \includegraphics[width=\textwidth]{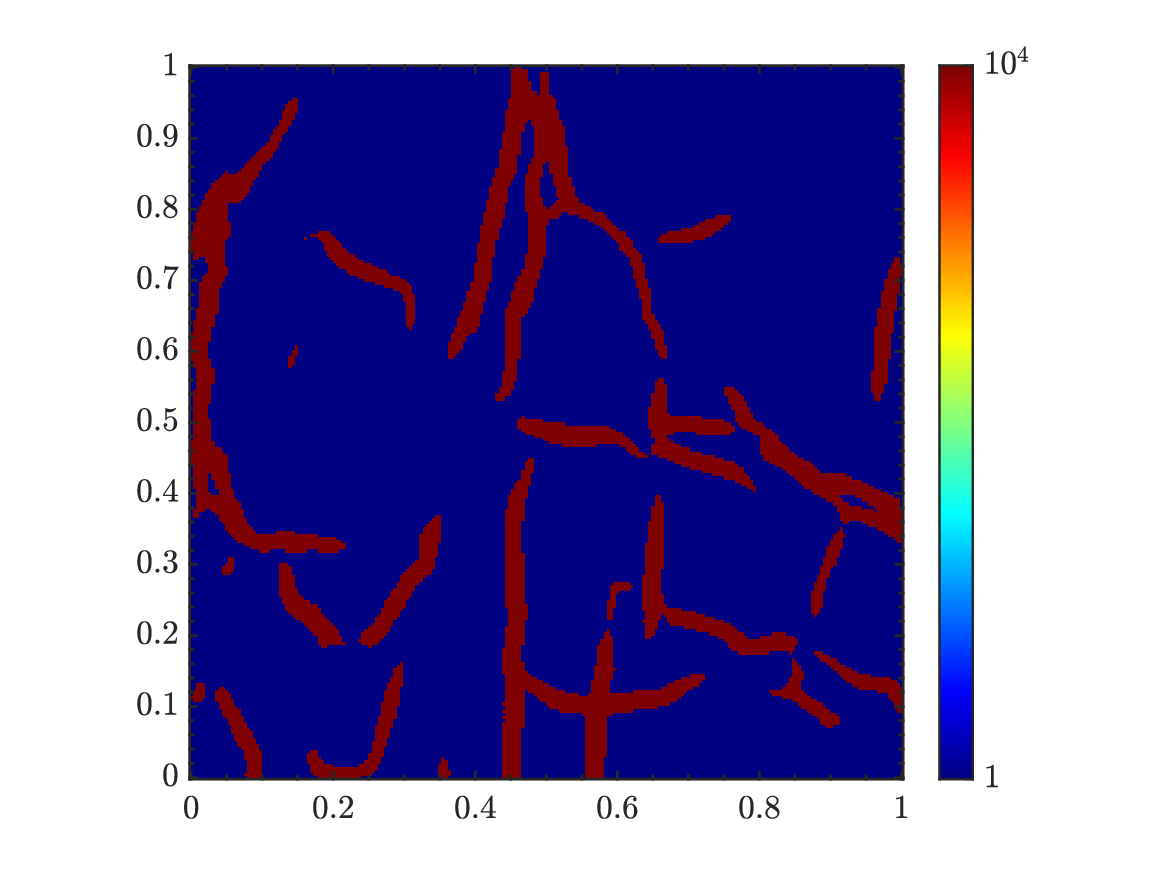}
        \caption{$\kappa_3$}
    \end{subfigure}
    \caption{Three realistic fractured permeability fields generated from the GeoCrack dataset.}
    \label{fig:ex3_kappa}
\end{figure}

The spatial domain, coarse and fine grid resolutions, and temporal discretization parameters remain consistent with those established in section \ref{sec:ex1}. We evaluate the spatial approximation accuracy by comparing the Energy and $L^2$ errors for the LSSI method at iteration step $n=4$ against the LKSI method across iteration steps ranging from $n=1$ to $n=8$. For all cases, the oversampling parameter is fixed at $m=4$.

\cref{tab:fracture_results} summarizes the numerical results across the three different fracture configurations. Consistent with our previous observations, the LKSI method exhibits robust and rapid convergence as the iteration step $n$ increases. Notably, the LKSI method consistently achieves comparable or superior accuracy to LSSI-4 with fewer iteration steps. For instance, in terms of Energy error, LKSI surpasses the accuracy of LSSI-4 at $n=4$ for $\kappa_1$, at $n=3$ for $\kappa_2$, and as early as $n=2$ for $\kappa_3$. As the iterations continue up to $n=8$, the errors for LKSI further decrease significantly. 
These results substantiate that the proposed LSI framework—and the LKSI method in particular—provides a highly effective, stable, and computationally efficient approach for simulating transient diffusion in complex, real-world multiscale environments.
This experiment confirms that the proposed LSI framework—and particularly the LKSI method—is highly effective, stable, and computationally efficient for solving parabolic equations in complex, real-world multiscale environments.

\begin{table}[htbp]
\centering
\caption{Comparison of Energy and $L^2$ errors for LSSI-4 and LKSI (iterations 1 to 8) methods on realistic fractured media.}
\label{tab:fracture_results}
\resizebox{\textwidth}{!}{
\begin{tabular}{cc|c|cccccccc}
\toprule
\multirow{2}{*}{\textbf{Field}} & \multirow{2}{*}{\textbf{Error}} & \textbf{LSSI} & \multicolumn{8}{c}{\textbf{LKSI}} \\
\cmidrule(lr){3-3} \cmidrule(lr){4-11}
 & & $n=4$ & $n=1$ & $n=2$ & $n=3$ & $n=4$ & $n=5$ & $n=6$ & $n=7$ & $n=8$ \\
\midrule
\multirow{2}{*}{$\kappa_1$} & Energy & 5.76e-03 & 3.16e-02 & 1.84e-02 & 5.99e-03 & 5.03e-03 & 4.75e-03 & 4.49e-03 & 4.32e-03 & 4.15e-03 \\
 & $L^2$ & 2.49e-04 & 2.80e-03 & 1.16e-03 & 2.73e-04 & 1.86e-04 & 1.56e-04 & 1.28e-04 & 1.08e-04 & 8.88e-05 \\
\midrule
\multirow{2}{*}{$\kappa_2$} & Energy & 1.26e-02 & 3.61e-02 & 2.25e-02 & 1.12e-02 & 9.92e-03 & 9.00e-03 & 8.09e-03 & 6.74e-03 & 6.19e-03 \\
 & $L^2$ & 9.74e-04 & 4.49e-03 & 2.09e-03 & 9.43e-04 & 7.69e-04 & 6.42e-04 & 5.04e-04 & 3.17e-04 & 2.31e-04 \\
\midrule
\multirow{2}{*}{$\kappa_3$} & Energy & 3.48e-02 & 4.01e-02 & 2.97e-02 & 2.34e-02 & 2.04e-02 & 1.82e-02 & 1.50e-02 & 1.32e-02 & 1.18e-02 \\
 & $L^2$ & 3.90e-03 & 6.26e-03 & 4.45e-03 & 3.00e-03 & 2.31e-03 & 1.85e-03 & 1.24e-03 & 9.52e-04 & 6.78e-04 \\
\bottomrule
\end{tabular}
}
\end{table}

\section{Conclusions}
\label{sec:conclu}
In this paper, we have extended the Localized Subspace Iteration (LSI) framework to solve time-dependent multiscale parabolic equations. By iteratively approximating the dominant eigenspaces of local inverse operators, the proposed method constructs optimal, low-dimensional multiscale trial spaces that effectively capture the essential slow-decaying low-frequency modes of the solution. To overcome the severe time-step restrictions imposed by highly heterogeneous media, we decoupled the basis construction into an offline stage and introduced a contrast-independent partially explicit temporal splitting scheme for the online stage. This splitting strategy explicitly evolves the dominant low-frequency modes while implicitly treating the high-frequency microscopic corrections. Consequently, it guarantees unconditional stability with respect to high-contrast coefficients and significantly reduces the computational burden associated with global implicit inversions. We established rigorous a priori error estimates in both the energy and $L^2$ norms. Comprehensive numerical experiments, including simulations on realistic fracture networks from the GeoCrack dataset, validated our theoretical findings.
The results demonstrate that the LSI framework, particularly the LKSI variant, achieves exceptional accuracy, convergence, and substantial computational efficiency, establishing it as a highly capable numerical tool for complex multiscale simulations.



\bibliographystyle{siamplain}
\bibliography{references}

@article{Hornung1990,
  author = {Hornung, Ulrich and Showalter, Ralph E},
  title = {Diffusion models for fractured media},
  journal = {J. Math. Anal. Appl.},
  volume = {147},
  number = {1},
  pages = {69--80},
  year = {1990}
}

@article{Biot1962,
  author = {Biot, M. A.},
  title = {Mechanics of deformation and acoustic propagation in porous media},
  journal = {J. Appl. Phys.},
  volume = {33},
  number = {4},
  pages = {1482--1498},
  year = {1962}
}

@book{Chen2006,
  author = {Chen, Zhangxin and Huan, Guanren and Ma, Yuanle},
  title = {Computational Methods for Multiphase Flows in Porous Media},
  publisher = {SIAM},
  address = {Philadelphia},
  year = {2006}
}

@article{Babuska1983,
  author = {Babu{\v{s}}ka, Ivo and Osborn, John E.},
  title = {Generalized finite element methods: their performance and their relation to mixed methods},
  journal = {SIAM J. Numer. Anal.},
  volume = {20},
  number = {3},
  pages = {510--536},
  year = {1983}
}

@article{Hou1997,
  author = {Hou, Thomas Yizhao and Wu, Xiao-Hui},
  title = {A multiscale finite element method for elliptic problems in composite materials and porous media},
  journal = {J. Comput. Phys.},
  volume = {134},
  number = {1},
  pages = {169--189},
  year = {1997}
}

@article{Hou1999,
  author = {Hou, Thomas Yizhao and Wu, Xiao-Hui and Cai, Zhiqiang},
  title = {Convergence of a multiscale finite element method for elliptic problems with rapidly oscillating coefficients},
  journal = {Math. Comp.},
  volume = {68},
  pages = {913--943},
  year = {1999}
}

@article{Malqvist2014,
  author = {M{\aa}lqvist, Axel and Peterseim, Daniel},
  title = {Localization of elliptic multiscale problems},
  journal = {Math. Comp.},
  volume = {83},
  number = {290},
  pages = {2583--2603},
  year = {2014}
}

@book{Malqvist2020,
  author = {M{\aa}lqvist, Axel and Peterseim, Daniel},
  title = {Numerical homogenization by localized orthogonal decomposition},
  publisher = {SIAM},
  year = {2020}
}

@article{Efendiev2013,
  author = {Efendiev, Yalchin and Galvis, Juan and Hou, Thomas Yizhao},
  title = {Generalized multiscale finite element methods ({GMsFEM})},
  journal = {J. Comput. Phys.},
  volume = {251},
  pages = {116--135},
  year = {2013}
}

@article{Babuska2011,
  author = {Babu{\v{s}}ka, Ivo and Lipton, Robert},
  title = {Optimal local approximation spaces for generalized finite element methods with application to multiscale problems},
  journal = {Multiscale Model. Simul.},
  volume = {9},
  number = {1},
  pages = {373--406},
  year = {2011}
}

@article{Chung2016,
  author = {Chung, Eric T. and Efendiev, Yalchin and Leung, Wing Tat and Ye, S.},
  title = {Generalized multiscale finite element methods for space-time heterogeneous parabolic equations},
  journal = {Multiscale Model. Simul.},
  volume = {14},
  number = {2},
  pages = {712--737},
  year = {2016}
}

@article{Li2019,
  author = {Li, M. and Chung, Eric T. and Jiang, Lijian},
  title = {A constraint energy minimizing generalized multiscale finite element method for parabolic equations},
  journal = {Multiscale Model. Simul.},
  volume = {17},
  number = {3},
  pages = {996--1018},
  year = {2019}
}

@article{Guan2024,
  author = {Guan, Xiaofei and Jiang, Lijian and Wang, Yajun},
  title = {Regularized coupling multiscale method for thermomechanical coupled problems},
  journal = {J. Comput. Phys.},
  volume = {499},
  pages = {112737},
  year = {2024}
}

@article{Guan2026,
  author = {Guan, Xiaofei and Jiang, Lijian and Wang, Yajun and Yang, Zihao},
  title = {Localized subspace iteration methods for multiscale problems},
  journal = {J. Comput. Phys.},
  volume = {547},
  pages = {114559},
  year = {2026}
}

@book{Bensoussan1978,
  author = {Bensoussan, Alain and Lions, Jacques-Louis and Papanicolaou, George},
  title = {Asymptotic Analysis for Periodic Structures},
  publisher = {North-Holland},
  address = {Amsterdam},
  year = {1978}
}

@article{Allaire1992,
  author = {Allaire, Gr{\'e}goire},
  title = {Homogenization and two-scale convergence},
  journal = {SIAM J. Math. Anal.},
  volume = {23},
  number = {6},
  pages = {1482--1518},
  year = {1992}
}

@article{Hughes1998,
  author = {Hughes, Thomas J. R. and Feij{\'o}o, G. R. and Mazzei, L. and Quincy, J.-B.},
  title = {The variational multiscale method---a paradigm for computational mechanics},
  journal = {Comput. Methods Appl. Mech. Eng.},
  volume = {166},
  number = {1--2},
  pages = {3--24},
  year = {1998}
}

@article{E2003,
  author = {{E}, Weinan and Engquist, Bj{\"o}rn},
  title = {The heterogeneous multiscale methods},
  journal = {Commun. Math. Sci.},
  volume = {1},
  number = {1},
  pages = {87--132},
  year = {2003}
}

@book{saad2011numerical,
  title={Numerical methods for large eigenvalue problems: revised edition},
  author={Saad, Yousef},
  year={2011},
  publisher={SIAM}
}

@article{chung2021contrast,
  title={Contrast-independent partially explicit time discretizations for multiscale flow problems},
  author={Chung, Eric T. and Efendiev, Yalchin and Leung, Wing Tat and Vabishchevich, Petr N},
  journal={J. Comput. Phys.},
  volume={445},
  pages={110578},
  year={2021},
  publisher={Elsevier}
}

@misc{ansari2024geocrack,
  author = {Ansari, M. Y.},
  title = {{GeoCrack}: A High-Resolution Dataset of Fracture Edges in Geological Outcrops},
  year = {2024},
  publisher = {Harvard Dataverse},
  doi = {10.7910/DVN/E4OXHQ}
}

@article{liu2024numerical,
  title={Numerical simulation of multiphase multi-physics flow in underground reservoirs: Frontiers and challenges},
  author={Liu, Piyang and Zhao, Jianlin and Li, Zheng and Wang, Han},
  journal={Capillarity},
  volume={12},
  number={3},
  pages={72--79},
  year={2024}
}

@article{lyu2025multiscale,
  title={Multiscale modeling for multiphase flow and reactive mass transport in subsurface energy storage: A review},
  author={Lyu, Xiaocong and Wang, Wendong and Voskov, Denis and Liu, Piyang and Chen, Li},
  journal={Adv. Geo-Energy Res.},
  volume={15},
  number={3},
  pages={245--260},
  year={2025}
}

@article{dong2023multi,
  title={Multi-scale computational method for nonlinear dynamic thermo-mechanical problems of composite materials with temperature-dependent properties},
  author={Dong, Hao and Cui, Junzhi and Nie, Yufeng and Ma, Ruyun and Jin, Ke and Huang, Dongmei},
  journal={Commun. Nonlinear Sci. Numer. Simul.},
  volume={118},
  pages={107000},
  year={2023},
  publisher={Elsevier}
}

@article{abdulle2012heterogeneous,
  title={The heterogeneous multiscale method},
  author={Abdulle, Assyr and {E}, Weinan and Engquist, Bj{\"o}rn and Vanden-Eijnden, Eric},
  journal={Acta Numer.},
  volume={21},
  pages={1--87},
  year={2012},
  publisher={Cambridge University Press}
}

@article{hughes2007variational,
  title={Variational multiscale analysis: the fine-scale {Green}'s function, projection, optimization, localization, and stabilized methods},
  author={Hughes, Thomas J. R. and Sangalli, Giancarlo},
  journal={SIAM J. Numer. Anal.},
  volume={45},
  number={2},
  pages={539--557},
  year={2007},
  publisher={SIAM}
}

@article{ma2022novel,
  title={Novel design and analysis of generalized finite element methods based on locally optimal spectral approximations},
  author={Ma, Chupeng and Scheichl, Robert and Dodwell, Tim},
  journal={SIAM J. Numer. Anal.},
  volume={60},
  number={1},
  pages={244--273},
  year={2022},
  publisher={SIAM}
}

@book{chung2023multiscale,
  title={Multiscale Model Reduction: Multiscale Finite Element Methods and Their Generalizations},
  author={Chung, Eric T. and Efendiev, Yalchin and Hou, Thomas Yizhao},
  series={Applied Mathematical Sciences},
  volume={212},
  year={2023},
  publisher={Springer}
}

@article{chung2018constraint,
  title={Constraint energy minimizing generalized multiscale finite element method},
  author={Chung, Eric T. and Efendiev, Yalchin and Leung, Wing Tat},
  journal={Comput. Methods Appl. Mech. Eng.},
  volume={339},
  pages={298--319},
  year={2018},
  publisher={Elsevier}
}

@article{brahim1992correctors,
  title={Correctors for the homogenization of the wave and heat equations},
  author={Brahim-Otsmane, Safia and Francfort, Gilles A and Murat, Fran{\c{c}}ois},
  journal={J. Math. Pures Appl.},
  volume={71},
  number={3},
  pages={197--231},
  year={1992}
}

@article{jiang2007multiscale,
  title={Multiscale methods for parabolic equations with continuum spatial scales},
  author={Jiang, Lijian and Efendiev, Yalchin and Ginting, Victor},
  journal={Discrete Contin. Dyn. Syst. Ser. B},
  volume={8},
  number={4},
  pages={833--859},
  year={2007},
  publisher={Discrete Contin. Dyn. Syst. Ser. B}
}

@article{john2006two,
  title={A two-level variational multiscale method for convection-dominated convection--diffusion equations},
  author={John, Volker and Kaya, Songul and Layton, William},
  journal={Comput. Methods Appl. Mech. Eng.},
  volume={195},
  number={33--36},
  pages={4594--4603},
  year={2006},
  publisher={Elsevier}
}

@article{ming2007analysis,
  title={Analysis of the heterogeneous multiscale method for parabolic homogenization problems},
  author={Ming, Pingbing and Zhang, Pingwen},
  journal={Math. Comp.},
  volume={76},
  number={257},
  pages={153--177},
  year={2007}
}

@article{maalqvist2018multiscale,
  title={Multiscale techniques for parabolic equations},
  author={M{\aa}lqvist, Axel and Persson, Anna},
  journal={Numer. Math.},
  volume={138},
  number={1},
  pages={191--217},
  year={2018},
  publisher={Springer}
}

@article{ascher1995implicit,
  title={Implicit-explicit methods for time-dependent partial differential equations},
  author={Ascher, Uri M. and Ruuth, Steven J. and Wetton, Brian T. R.},
  journal={SIAM J. Numer. Anal.},
  volume={32},
  number={3},
  pages={797--823},
  year={1995},
  publisher={SIAM}
}

@book{evans2022partial,
  title={Partial Differential Equations},
  author={Evans, Lawrence C.},
  volume={19},
  year={2022},
  publisher={American Mathematical Society}
}

@article{galvis2010domain,
  title={Domain decomposition preconditioners for multiscale flows in high contrast media: reduced dimension coarse spaces},
  author={Galvis, Juan and Efendiev, Yalchin},
  journal={Multiscale Modeling \& Simulation},
  volume={8},
  number={5},
  pages={1621--1644},
  year={2010},
  publisher={SIAM}
}
\end{document}